\documentclass[a4paper,11pt]{amsart}
\usepackage[all,arc]{xy}
\usepackage[all]{xy}
\usepackage{latexsym}
\usepackage{amssymb}
\usepackage{amsmath}
\usepackage{tikz}
\usepackage{graphicx}
\usepackage{array,booktabs}
\newtheorem{definition}{Definition}[section]
\newtheorem{lemma}[definition]{Lemma}
\newtheorem{prop}[definition]{Proposition}
\newtheorem{theorem}[definition]{Theorem}
\newtheorem{cor}[definition]{Corollary}

\theoremstyle{definition}
\newtheorem{fact}[definition]{Fact}

\newcommand*{\AR}{\mathop{{\rm AR}}\nolimits}

\newcommand*{\CB}{\mathop{{\rm CB}}\nolimits}
\newcommand*{\CM}{\mathop{{\rm CM}}\nolimits}

\newcommand*{\End}{\mathop{{\rm End}}\nolimits}

\newcommand*{\Ext}{\mathop{{\rm Ext}}\nolimits}
\newcommand*{\Hom}{\mathop{{\rm Hom}}\nolimits}

\newcommand*{\nil}{\mathop{{\rm nil}}\nolimits}

\newcommand{\PE}{\mathop{{\rm PE}}\nolimits}

\newcommand*{\rad}{\mathop{{\rm rad}}\nolimits}

\newcommand*{\Soc}{\mathop{{\rm Soc}}\nolimits}

\newcommand{\Zg}{\mathop{{\rm Zg}}\nolimits}

\newcommand{\ZCM}{\mathop{{\rm ZCM}}\nolimits}

\newcommand*{\ex}{\exists}
\newcommand*{\fa}{\forall}

\newcommand*{\ih}{\imath}
\newcommand*{\ms}{\models}

\newcommand*{\ov}{\overline}
\newcommand*{\seq}{\subseteq}

\newcommand*{\sm}{\setminus}

\newcommand*{\bsm}{\left(\begin{smallmatrix}}
\newcommand*{\esm}{\end{smallmatrix}\right)}
\newcommand*{\bp}{\begin{pmatrix}}
\newcommand*{\ep}{\end{pmatrix}}

\newcommand*{\fty}{\infty}

\newcommand*{\wg}{\wedge}

\newcommand*{\wt}{\widetilde}

\renewcommand{\emph}{\textbf}

\newcommand*{\xr}{\xrightarrow}

\newcommand*{\mD}{\mathcal{D}}
\newcommand*{\mm}{\mathbf{m}}

\newcommand*{\Z}{\mathbb{Z}}

\newcommand*{\al}{\alpha}
\newcommand*{\be}{\beta}

\newcommand*{\lam}{\lambda}

\newcommand*{\om}{\omega}

\renewcommand*{\phi}{\varphi}

\begin{document}

\footskip=30pt

\date{}

\title[]{The Ziegler spectrum of the $D$-infinity plane singularity}

\author[]{Inna Los and Gena Puninski}

\address{Belarusian State University, Faculty of Mechanics and Mathematics, av. Nezalezhnosti 4,
Minsk 220030, Belarus}

\email{innalos.los1@gmail.com, punins@mail.ru}



\begin{abstract}
We will describe the Cohen--Macaulay part of the Ziegler spectrum of the $D_{\fty}$ plane singularity $S$ and
calculate the nilpotency index of the radical of the category of finitely generated Cohen--Macaulay $S$-modules.
\end{abstract}

\maketitle

\section{Introduction}\label{S-intro}

A local commutative noetherian ring is said to be of \emph{countable} (or \emph{discrete}) 
\emph{Cohen--Macaulay representation type} if it has only countably (infinitely) many non-isomorphic finitely 
generated maximal Cohen--Macaulay modules. For complete 1-dimensional surfaces the classification of \cite{BGS} 
shows that there are essentially two possibilities: $R= F[[x,y]]/(x^2)$, the $A_{\fty}$ plane singularity; and 
$S= F[[x,y]]/(x^2y)$, the $D_{\fty}$ plane singularity.

The model theory of Cohen--Macaulay $R$-modules was investigated in \cite{Pun16}. In particular the Cohen--Macaulay
part of the Ziegler spectrum of $R$, $\ZCM_R$, was described, and the nilpotency index of the radical of the
category of finitely generated Cohen--Macaulay modules was calculated. In this paper we will address similar
questions for the ring $S$. By factoring out $x^2$ we obtain a natural surjection $S\to R$, hence an embedding
from $\ZCM_R$ onto a closed subset of $\ZCM_S$. Our initial feeling was that the Cantor--Bendixson rank of points
in the former space will jump to higher values via this embedding (see \cite{LPP} or \cite{P-T14} for dealing
with this effect). However this is not the case --- we will show that the points preserve their ranks in the
ambient space; and the Cantor--Bendixson rank of $\ZCM_S$ equals 2.

From point of view of properties we will investigate in this paper, the category of Cohen--Macaulay modules over
$S$ is a 'double' of the corresponding category of $R$-modules. For instance the Krull--Gabriel dimension of the
definable category generated by finitely generated Cohen--Macaulay modules equals 2 for $R$ and for $S$. Further,
similarly to the $A_{\fty}$-case, one can glue from the Auslander--Reiten quiver of the category of finitely generated
Cohen--Macaulay $S$-modules the Ringel quilt, which is the M\"obius stripe.

This gives a global geometric structure to the above category including objects, but also morphisms, - the part
which is so often neglected. Using this realization we will show that the nilpotency index of the radical of this
category equals $\om\cdot 2$. Note that, by Schr\"oer \cite[Prop. 6.1]{Schro}, for finite dimensional algebras
this index cannot be equal to a limit ordinal.

To classify points of the space $\ZCM_S$, i.e. indecomposable pure injective Cohen--Macaulay modules, we will
employ the so-called 'interval method' which is based on Ziegler's results \cite{Zieg}. Namely each point in
this space is uniquely determined by a nontrivial cut it creates on a given interval in the lattice of positive
primitive formulae. Taking intervals means to take open sets covering the category of finitely generated
Cohen--Macaulay modules considered as a surface. We will show that two basic open sets are sufficient to detect
all points of this space and calculate their bases of open sets in Ziegler topology.

Due to many technicalities involved in proofs, we will be quite concise when executing them in detail. Say, the
description of irreducible morphisms between finitely generated Cohen--Macaulay modules over $S$ is known to experts
in this area, but no written account seems to exist. So we rely on 'believe or check' approach when explaining this
matter. Further, because the model theoretic part of this paper is similar to the $A_{\fty}$-case, which was
carefully explained in \cite{Pun16}, we will allow themselves to be quite sketchy on some instances.

The point of this paper is to propel the method rather than calculations. It proved to be successful for some
classes of finite dimensional algebras and some Cohen--Macaulay singularities. The reason for this is obvious: both
finite dimensional modules, and finitely generated Cohen--Macaulay modules over complete Cohen--Macaulay rings are
pure injective.

\section{The ring}\label{S-ring}

In this part we will introduce the object of our interest, the ring $S$, and mention its basic properties which 
will be used in the paper. Let $F$ be an algebraically closed field whose characteristic is different from 2, 
and let $S= F[[x,y]]/(x^2y)$ be the factor of the ring of power series $F[[x,y]]$ by the ideal generated by $x^2 y$,
the so-called \emph{$D_{\fty}$ plane singularity}. Many properties of $S$ can be extracted from fairly general
commutative algebra.

In particular, $S$ is a local complete 1-dimensional commutative noetherian ring with zero divisors. Namely
if $P$ is a prime ideal of $S$ then $x^2y= 0$ yields $x\in P$ or $y\in P$ (or both). It follows easily that
$P= xS$, $P= yS$ or $P= \mm= (x,y)$, the maximal ideal of $S$.

Note that there is a natural isomorphism $S/x^2S\cong R= F[[x,y]]/(x^2)$, therefore every $R$-module is an
$S$-module. Further $x^2S\cong S/yS\cong F[[x]]$, hence writing $F[[x]]$ or $F((x))$, the field of Laurent
power series, we mean the corresponding $S$-module. Also $xyS\cong S/xS\cong F[[y]]$, hence similar applies to
$F[[y]]$ and $F((y))$.

The ring theoretical properties of $R$ and its modules can be found in \cite{Pun16}, - we will use them freely.
For instance the quotient ring $Q_R$ of $R$ is a uniform module, and the integral closure $\wt R$ of $R$ in
$Q_R$ is a non-noetherian valuation ring of Krull dimension one. Further $\wt R$ can be identified with
$R+ xQ_R$.

Because $S$ is 1-dimensional, it suffices to invert any nonzero divisor, say $x+y$, to obtain its ring of quotients
$Q_S$, for instance $Q_S= S[(x+y)^{-1}]$. Namely $T= (x+y)^{-k}$, $k= 1, 2,\dots$ is a multiplicative closed set,
and the localization $ST^{-1}$ has only two prime ideals $xST^{-1}$ and $yST^{-1}$, both consisting
of zero divisors. It follows that each nonzero divisor in this ring is invertible, hence $ST^{-1}= Q_S$.

For basic facts in the theory of Cohen--Macaulay rings and modules see \cite{L-W} or \cite{Yosh}. We will use
the homological definition of this notion. Because $S$ is 1-dimensional, a finitely generated $S$-module $M$ is
said to be \emph{Cohen--Macaulay}, $\CM$, if $\Hom(F,M)=0$. Clearly this is the same as $\Soc(M)=0$, i.e.
$mx= my=0$ for $m\in M$ implies $m=0$. For instance the ring $S$ itself is \emph{Cohen--Macaulay}.

Furthermore $S$ is \emph{Gorenstein}, that is of injective dimension $1$. Namely $x+y$ is a nonzero divisor, and
the ring $S/(x+y)S\cong F[[x]]/(x^3)$ has simple socle, hence (see \cite[Prop. 21.5]{Eis}) is zero-dimensional
Gorenstein. Then $S$ is Gorenstein by \cite[Prop. 3.1.19 b)]{B-H}.

Also $S$ has uniform dimension $2$, therefore its injective envelope $I(S)$ is a decomposable module. In fact,
because $S$ is noetherian, by \cite[Prop. 2.9]{B-H}, we have the following injective resolution.

$$
0\to S\to I(S/xS)\oplus I(S/yS)\to I(S/\mm)\to 0\,.
$$

From $S/yS\cong F[[x]]$ and $x^2y=0$ it follows that $I(S/yS)$ is annihilated by $y$, hence $I(S/yS)\cong F((x))$.
Similarly $S/xS\cong F[[y]]$ implies that $x^2$ annihilates $I(S/xS)$, therefore $I(S/xS)\cong Q_R$, the quotient
ring of $R$.

Comparing this with the injective resolution $0\to S\to Q_S\to Q_S/S\to 0$ (see Matlis \cite[Thm. 13.1]{Mat}) we
conclude that $Q_S\cong F((x))\oplus Q_R$ and $Q_S/S\cong I(S/\mm)$, the minimal injective cogenerator in the
category of $S$-modules.

By $\wt S$ we will denote the integral closure of $S$ in $Q_S$. If $u= x(x+y)^{-1}$ then $u^2= u^3$, hence
$u\in \wt S$, and $e= u^2$ is an idempotent in $\wt S$. Now we can see the above decomposition of $Q_S$ more clearly.

\begin{lemma}\label{qs-dec}
$Q_S= x^2 Q_S\oplus yQ_S$. Further $x^2 Q_S= eQ_S\cong F((x))$ and $yQ_S= (1-e) Q_S\cong Q_R$.
\end{lemma}
\begin{proof}
Since $e\in Q_S$ is an idempotent, we have $Q_S= eQ_S\oplus (1-e)Q_S$. Clearly $eQ_S= x^2Q_S\cong F((x))$. Further
$1-e= y(2x+y)(x+y)^{-2}\in yQ_S$, therefore $(1-e)Q_S\seq yQ_S$. Also $ye=0$ implies $y\in (1-e)Q_S$, from which
the inverse inclusion follows.

Finally $yQ_S\cong Q_R$ by sending $y$ to $1$.
\end{proof}

We will consider $Q_S$ as the \emph{universal domain} which contains all objects of our interest. In particular
we identify $R_S$ with $yS\subset Q_S$; and $Q_R$, considered as an $S$-module, with $yQ_S$.

Recall that an \emph{overring} of $S$ is any subring of the integral closure $\wt S$ containing $S$. Since the
ideal  $\mm^2$ requires 3 generators, it follows (see \cite[Thm. 4.18]{L-W}) that $S$ is not \emph{Bass}, i.e.
there is an overring of $S$ which is not Gorenstein. We will pinpoint such an overring.

\begin{lemma}\label{s-pr}
There is the least proper overring $S'= \End_S(\mm)$ of $S$ generated over $S$ by $z= x^2(x+y)^{-1}$. Further $S'$
is a local non-Gorenstein ring which coincides with the endomorphism ring of $S'$ considered as an $S$-module.
\end{lemma}
\begin{proof}
Because $S$ is Gorenstein, it follows (this fact is called the rejection lemma in Kiev school) that $S$ has the
least proper overring $S'$. It can be identified with the endomorphism ring of $\mm$: $S'= S[z]$ where
$z= x^2(x+y)^{-1}\in \wt S$. For instance multiplication by $z$ defines the endomorphism of $\mm$ which sends $x$ to
$x^2$ and $y$ to zero.

In fact (see \cite[Exam. 14.23]{L-W}) $S'$ is isomorphic to the ring obtained from $F[[x,y,z]]$ by imposing
relations $z^2= zx=x^2$ and $yz=0$, the so-called \emph{$E_{\fty}$-singularity}. Clearly $S'$ is a local ring whose
Jacobson radical is generated by $x, y, z$.

Further $S'$ is no longer Gorenstein. Namely $x+y$ is a nonzero divisor in $S'$, and $S'/(x+y)S'$ is isomorphic to the
ring obtained from $F[[x,z]]$ by imposing relations $z^2= zx=x^2=0$. This ring has a $2$-dimensional socle, hence is
not Gorenstein, therefore the same holds true for $S'$.

Since the image of each $f\in \End_S(S')$ is uniquely determined by $f(1)\in S'$ it follows that $\End_S(S')$
coincides with $S'$, in particular $S'_S$ is an indecomposable module.
\end{proof}

Next we will describe the integral closure of $S$. Recall that $u= x(x+y)^{-1}\in \wt S$ and $e= u^2= x^2(x+y)^{-2}$
is an idempotent.

\begin{lemma}\label{int}
$\wt S= S[u]+ xy Q_S$.
\end{lemma}
\begin{proof}
Clearly $xyQ_S\seq \wt S$, because this ideal is nilpotent, therefore $S[u]+ xyQ_S\seq\wt S$.

To prove the converse inclusion note that $u= x(x+y)^{-1}\in \wt S$ implies $x^k (x+y)^{-k}\in \wt S$;
and similarly $y^k(x+y)^{-k}\in \wt S$. Further clearly $x(x+y)^{-2}\notin S$. By multiplying this by $e$ we
conclude that $x^k(x+y)^{-l}$ is not in $\wt S$ as soon as $1\leq k< l$, and the same holds true for $y^k(x+y)^{-l}$.
It remains to show that each nonzero element $(\sum_{i<k} \al_i x^i+ \sum_{j< k}\be_j y^j)(x+y)^{-k}$ is not in
$\wt S$ if $k\geq 1$, but this is easily checked.
\end{proof}

From this we obtain the following decomposition of $\wt S_S$.

\begin{lemma}\label{st-dec}
$\wt S_S\cong e\wt S\oplus (1-e)\wt S$, where $e\wt S= eS$ is isomorphic to $x^2 S= F[[x]]$ and
$(1-e)\wt S= y^2(x+y)^{-1}S+ xy Q_S$ is isomorphic to $\wt R$ (as $S$-modules).
\end{lemma}
\begin{proof}
Since $exy=0$ and $eu= e$, by Lemma \ref{int} we conclude that $e\wt S= eS$. Also $eS= x^2(x+y)^{-2}S$ is isomorphic
to $x^2S$ as an $S$-module.

Further $1-e= y(2x+y)(x+y)^{-2}$ equals $z'= y^2(x+y)^{-2}$ modulo $xyQ_S$. Since $z'u\in xyQ_S$, we obtain
$(1-e)\wt S= z'S+ xy Q_S$. It is easily checked that the map $z'\mapsto 1$ extends uniquely to an isomorphism from
this submodule onto $\wt R_S$.
\end{proof}

It will be easier to deal with another copy of $\wt R$ within $Q_S$. Namely is its easily shown that the submodule
$yS+ xy Q_S$ is isomorphic to $\wt R_S$ via the map $y\mapsto 1$. We will identify $\wt R_S$ with this submodule.

We extend the above definition of finitely generated Cohen--Macaulay modules to arbitrary modules. Namely
an $S$-module $M$ is said to be \emph{Cohen--Macaulay} if $\Hom(F,M)=0$, i.e. if $\Soc(M)=0$.

Clearly $\CM$-modules over $S$ form a \emph{definable class (category)}, i.e. (see \cite{CB98} or \cite{Kra}) this
class is closed with respect to direct products, direct limits and pure submodules. It is defined by the first
order sentence $\fa\, x\, (vx=0\wg vy=0\to v=0)$, where $v$ denotes a variable; therefore we obtain the
\emph{theory} $T_{CM}$ whose models are $\CM$-modules.

Recall (see \cite[Sect. 3.3.4]{Preb}) that there is a natural one-to-one correspondence between definable categories
of modules over any ring $U$ and first order theories of $U$-modules whose models are closed with respect to direct
products. We will use these terms as synonymous.

\section{Finitely generated $\CM$-modules}\label{S-fg}

In this section we recall the classification of indecomposable finitely generated $\CM$-modules over $S$. It can
be combined from \cite{Yosh, Schre, L-W}, for instance we will use notation of \cite[p. 78]{Yosh}. In particular
each indecomposable finitely generated $\CM$-module is isomorphic to either an ideal of $S$, or a submodule of
$S^2$. We will mention this realization, but prefer to define such modules using generators and relations.

\textbf{1)}. $S$, $A= x^2 S$, $B= yS$, $C= xyS$ and $D= xS$. Note that $Bx^2=Cx^2=0$, hence $B$ and $C$ are
$R$-modules.

\textbf{2)}. For each $k\geq 1$ the module $M_k$ is given by generators $m, n$ and relations $mx= ny^k$, $nx=0$;
for instance here is a diagram for $M_2$.

\vspace{3mm}

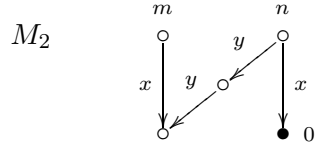
\begin{center}
$
M_2\hspace{10mm}
\xymatrix@C=14pt@R=10pt{%
*+={\circ}\ar[dd]_x\ar@{}+<0pt,10pt>*{_m}&&*+={\circ}\ar[ld]_y\ar[dd]^x\ar@{}+<0pt,10pt>*{_n}\\
&*+={\circ}\ar[ld]_(.3)y&\\
*+={\circ}&&*+={\bullet}\ar@{}+<10pt,0pt>*{_0}
}
$
\end{center}

\vspace{3mm}

It is easily seen that the map $m\mapsto y^{k+1}$ and $n\mapsto xy$ gives an isomorphism from $M_k$ onto the ideal
$(y^{k+1},xy)$ of $S$. Further $M$ is annihilated by $x^2$, hence is a $\CM$-module over $R$.

Also $M_k$ is isomorphic to the submodule of $Q_S$ generated by $m= y$ and $n= xy(x+y)^{-k}$.

\textbf{3)}. For each $k\geq 1$ the module $Y_k$ is given by generators $m, n$ and relations $mx= ny^k$, $nxy=0$;
here is a diagram for $Y_2$.

\vspace{3mm}

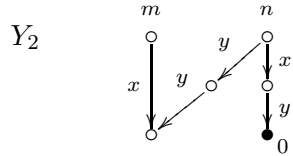
\begin{center}
$
Y_2\hspace{10mm}
\xymatrix@C=14pt@R=10pt{%
*+={\circ}\ar[dd]_x\ar@{}+<0pt,10pt>*{_m}&&*+={\circ}\ar[ld]_y\ar[d]^x\ar@{}+<0pt,10pt>*{_n}\\
&*+={\circ}\ar[ld]_(.25)y&*+={\circ}\ar[d]^y\\
*+={\circ}&&*+={\bullet}\ar@{}+<6pt,-4pt>*{_0}
}
$
\end{center}

\vspace{3mm}

For instance $Y_k$ is isomorphic to the ideal $(y^k,x)$ of $S$ via the map $m\mapsto y^k$ and $n\mapsto x$,
in particular $Y_1\cong \mm$.

\textbf{4)}. The module $X_k$, $k\geq 1$ is generated by $m, n$ with relations $mxy= ny^k$ and $nx=0$; here is
a diagram for $X_1$.

\vspace{3mm}

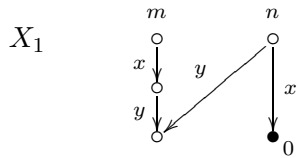
\begin{center}
$
X_1\hspace{10mm}
\xymatrix@C=14pt@R=10pt{%
*+={\circ}\ar[d]_x\ar@{}+<0pt,10pt>*{_m}&&*+={\circ}\ar[lldd]_y\ar[dd]^x\ar@{}+<0pt,10pt>*{_n}\\
*+={\circ}\ar[d]_y&&&\\
*+={\circ}&&*+={\bullet}\ar@{}+<6pt,-4pt>*{_0}
}
$
\end{center}

\vspace{3mm}

For instance $X_1$ is isomorphic to $S'_S$ via the map $m\mapsto 1$ and $n\mapsto x-z$.

Further $X_k$ is isomorphic to the submodule of $S^2$ generated by $m= (y^k,x)$ and $n= (xy,0)$; and can be
identified with the submodule of $Q_S$ generated by $m=1$ and $n= xy(x+y)^{-k}$.

\textbf{5)}. The module $N_k$, $k\geq 1$ is generated by $m, n$ with relations $mxy= ny^{k+1}$ and $nxy=0$;
here is the diagram for $N_1$.

\vspace{3mm}

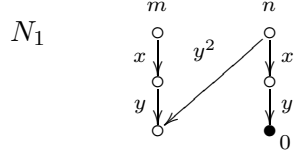
\begin{center}
$
N_1\hspace{10mm}
\xymatrix@C=14pt@R=10pt{%
*+={\circ}\ar[d]_x\ar@{}+<0pt,10pt>*{_m}&&*+={\circ}\ar[lldd]_(.4){y^2}\ar[d]^x\ar@{}+<0pt,10pt>*{_n}\\
*+={\circ}\ar[d]_y&&*+={\circ}\ar[d]^y\\
*+={\circ}&&*+={\bullet}\ar@{}+<6pt,-4pt>*{_0}
}
$
\end{center}

\vspace{3mm}

For instance $N_k$ is isomorphic to the submodule of $S^2$ generated by $m= (y^k,x)$ and $n=(x,0)$.

The following is the $\AR$-quiver of the category of finitely generated $\CM$-modules over $S$ extrapolated from
Yoshino \cite[p. 78]{Yosh} and \cite[p. 25]{Schre}, except of the precise form of irreducible maps.

\vspace{3mm}

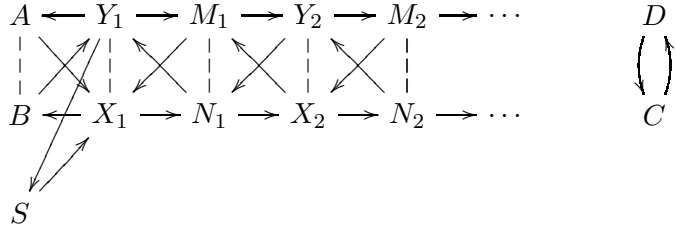
\begin{center}
$
\xymatrix@C=14pt@R=20pt{%
*+{A}\ar[rd]\ar@{--}[d]&*+{Y_1}\ar[l]\ar[r]\ar@{--}[d]\ar[ldd]&*+{M_1}\ar[dl]\ar[r]\ar@{--}[d]&
*+{Y_2}\ar[dl]\ar[r]\ar@{--}[d]&*+{M_2}\ar[dl]\ar[r]\ar@{--}[d]&*+{\dots}&&*+{D}\ar@/_{6pt}/[d]\\
*+{B}\ar[ur]&*+{X_1}\ar[l]\ar[r]&*+{N_1}\ar[lu]\ar[r]&*+{X_2}\ar[lu]\ar[r]&*+{N_2}\ar[ul]\ar[r]\ar@{--}[u]
&*+{\dots}&&*+{C}\ar@/_{6pt}/[u]\\
*+{S}\ar[ru]&&&&
}
$
\end{center}

\vspace{3mm}

The irreducible morphisms on this diagram resemble irreducible morphisms between string modules (see \cite{CB89})
given by graph maps: first factor, and then embed. They could be guessed and checked. Here is a complete list.

\textbf{1)}. $Y_k\to M_k$ is given by factoring out $nx$.

\vspace{2mm}

\textbf{2)}. $Y_k\to N_{k-1}$, $k\geq 2$ sends $m\mapsto my$ and $n\mapsto n$.

\vspace{2mm}

\textbf{3)}. $M_k\to Y_{k+1}$ sends $m\mapsto m$ and $n\mapsto ny$.

\vspace{2mm}

\textbf{4)}. $M_k\to X_k$ is given by $m\mapsto my$ and $n\mapsto n$.

\vspace{2mm}

\textbf{5)}. $X_k\to N_k$ is provided by $m\mapsto m$ and $n\mapsto ny$.

\vspace{2mm}

\textbf{6)}. $X_k\to M_{k-1}$, $k\geq 2$ is given by annihilating $mx- ny^{k-1}$.

\vspace{2mm}

\textbf{7)}. $N_k\to Y_k$ is provided by gluing $mx$ with $ny^k$.

\vspace{2mm}

\textbf{8)}. $N_k\to X_{k+1}$ is given by annihilating $nx$.

\vspace{4mm}

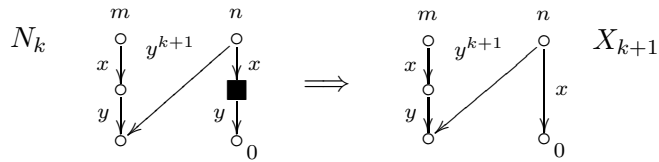
\begin{center}
\parbox{4cm}{%
$
N_k\hspace{5mm}
\xymatrix@C=14pt@R=10pt{%
*+={\circ}\ar[d]_x\ar@{}+<0pt,10pt>*{_m}&&*+={\circ}\ar[lldd]_(.3){y^{k+1}}\ar[d]^x\ar@{}+<0pt,10pt>*{_n}\\
*+={\circ}\ar[d]_y&&*+={\blacksquare}\ar[d]_y\\
*+={\circ}&&*+={\circ}\ar@{}+<6pt,-4pt>*{_0}
}
$
}
\hspace{-4mm}
$\Longrightarrow$
\hspace{3mm}
\parbox{4cm}{%
$
\xymatrix@C=14pt@R=10pt{%
*+={\circ}\ar[d]_x\ar@{}+<0pt,10pt>*{_m}&&*+={\circ}\ar[lldd]_(.3){y^{k+1}}\ar[dd]^x\ar@{}+<0pt,10pt>*{_n}\\
*+={\circ}\ar[d]_y&&&\\
*+={\circ}&&*+={\circ}\ar@{}+<6pt,-4pt>*{_0}
}
\hspace{-3mm} X_{k+1}
$
}
\end{center}

\vspace{4mm}

Here we show by the black square the element of the first module which is send to zero by this morphism.

\textbf{9)}. $B\to Y_1$ sends $y\mapsto m$, hence is the inclusion $yS\subset \mm$.

\vspace{2mm}

\textbf{10)}. $Y_1\to A$ is given by $n\mapsto x^2$ and $m\mapsto 0$. If we identify $Y_1$ with $\mm$ by setting
$m=y$ and $n=x$, then this map is given by multiplication by $z$.

\vspace{2mm}

\textbf{11)}. $X_1\to B$ is given by gluing $mx$ and $n$, hence by factoring out $z$ in $S'=X_1$.

\vspace{2mm}

\textbf{12)}. $S\to X_1$ sends $1\mapsto m$, hence is the inclusion $S\subset S'$.

\vspace{2mm}

\textbf{13)}. $Y_1\to S$ sends $m\mapsto y$ and $n\mapsto x$, hence is the inclusion $\mm\subset S$.

\vspace{2mm}

\textbf{14)}. $A\to X_1$ is given by sending $x^2\mapsto mx-n$.

\vspace{2mm}

\textbf{15)}. $C\to D$ is the inclusion and $D\to C$ is given by multiplication by $y$.

\vspace{2mm}

On Figure \ref{fig1} we show the unbridged version of the $\AR$-quiver, which includes the relations between
irreducible morphisms. We mark the finitely generated $\CM$-modules over $R$ by solid circles, - they comprise
roughly a quarter of the category of finitely generated $\CM$-modules over $S$.

\begin{figure}[t]
$$
\vcenter{%
\xymatrix@C=14pt@R=12pt{%
&&&&*+={\circ}\ar@{}+<0pt,10pt>*{_D}\ar@{}+<6pt,6pt>*{.}\ar@{}+<10pt,10pt>*{.}\ar@{}+<14pt,14pt>*{.}&&&&&\\
&&&*+={\bullet}\ar[ur]_{_{\subset}}\ar@{}+<0pt,10pt>*{_C}&&&&&&&\\
&&*+={\circ}\ar@{}+<0pt,10pt>*{_D}\ar[ur]_y&&&&
*+={\circ}\ar[dr]\ar@{}+<0pt,10pt>*{_{Y_3}}\ar@{}+<6pt,6pt>*{.}\ar@{}+<10pt,10pt>*{.}\ar@{}+<14pt,14pt>*{.}
\ar@{}+<-8pt,8pt>*{.}\ar@{}+<-14pt,14pt>*{.}\ar@{}+<-20pt,20pt>*{.}&&&&\\
&*+={\bullet}\ar@{}+<0pt,10pt>*{_C}\ar[ur]_{_{\subset}}&&&&
*+={\bullet}\ar[dr]\ar[ur]\ar@{}+<0pt,10pt>*{_{M_2}}\ar@{}+<-8pt,8pt>*{.}\ar@{}+<-14pt,14pt>*{.}\ar@{}+<-20pt,20pt>*{.}&&
*+={\circ}\ar[dr]\ar@{}+<0pt,10pt>*{_{N_2}}\ar@{}+<6pt,6pt>*{.}\ar@{}+<10pt,10pt>*{.}\ar@{}+<14pt,14pt>*{.}&&&&\\
*+={\circ}\ar@{}+<0pt,10pt>*{_D}\ar[ur]_y\ar@{}+<-6pt,-6pt>*{.}\ar@{}+<-10pt,-10pt>*{.}\ar@{}+<-14pt,-14pt>*{.}&&
&&*+={\circ}\ar[dr]\ar[ur]\ar@{}+<0pt,10pt>*{_{Y_2}}\ar@{}+<-8pt,8pt>*{.}\ar@{}+<-14pt,14pt>*{.}\ar@{}+<-20pt,20pt>*{.}
&&
*+={\circ}\ar[dr]\ar[ur]\ar@{}+<0pt,10pt>*{_{X_2}}&&
*+={\circ}\ar[dr]\ar@{}+<0pt,10pt>*{_{Y_2}}\ar@{}+<6pt,6pt>*{.}\ar@{}+<10pt,10pt>*{.}\ar@{}+<14pt,14pt>*{.}&&&\\
&&&*+={\bullet}\ar[dr]\ar[ur]\ar@{}+<0pt,10pt>*{_{M_1}}\ar@{}+<-8pt,8pt>*{.}\ar@{}+<-14pt,14pt>*{.}\ar@{}+<-20pt,20pt>*{.}&&
*+={\circ}\ar[dr]\ar[ur]\ar@{}+<0pt,10pt>*{_{N_1}}&&
*+={\bullet}\ar[dr]\ar[ur]\ar@{}+<0pt,10pt>*{_{M_1}}&&
*+={\circ}\ar[dr]\ar@{}+<0pt,10pt>*{_{N_1}}\ar@{}+<6pt,6pt>*{.}\ar@{}+<10pt,10pt>*{.}\ar@{}+<14pt,14pt>*{.}&&&\\
&&*+={\circ}\ar[dr]\ar[ur]\ar[r]\ar@{}+<0pt,10pt>*{_{Y_1}}
\ar@{}+<-10pt,-10pt>*{.}\ar@{}+<-14pt,-14pt>*{.}\ar@{}+<-6pt,-6pt>*{.}
\ar@{}+<-8pt,8pt>*{.}\ar@{}+<-14pt,14pt>*{.}\ar@{}+<-20pt,20pt>*{.}
&
*+={\circ}\ar[r]\ar@{}+<0pt,6pt>*{_S}&
*+={\circ}\ar[dr]\ar[ur]\ar@{}+<0pt,10pt>*{_{X_1}}&&*+={\circ}\ar[dr]\ar[ur]\ar[r]\ar@{}+<0pt,10pt>*{_{Y_1}}&
*+={\circ}\ar[r]\ar@{}+<0pt,6pt>*{_S}&*+={\circ}\ar[dr]\ar[ur]\ar@{}+<0pt,10pt>*{_{X_1}}&&
*+={\circ}\ar@{}+<0pt,10pt>*{_{Y_1}}\ar@{}+<6pt,0pt>*{.}\ar@{}+<10pt,0pt>*{.}\ar@{}+<14pt,0pt>*{.}&\\
&&&*+={\circ}\ar[ur]\ar@{}+<0pt,-10pt>*{_A}\ar@{}+<-6pt,0pt>*{.}\ar@{}+<-10pt,0pt>*{.}\ar@{}+<-14pt,0pt>*{.}&&
*+={\bullet}\ar[ur]\ar@{}+<0pt,-10pt>*{_B}&&
*+={\circ}\ar[ur]\ar@{}+<0pt,-10pt>*{_A}&&
*+={\bullet}\ar[ur]\ar@{}+<0pt,-10pt>*{_B}\ar@{}+<6pt,0pt>*{.}\ar@{}+<10pt,0pt>*{.}\ar@{}+<14pt,0pt>*{.}&&
}}
$$
\caption{}\label{fig1}
\end{figure}

\vspace{5mm}

For instance there are the following almost split sequences in the category of finitely generated $\CM$-modules.

$$
0\to M_k\to Y_{k+1}\oplus X_k\to N_k\to 0 \qquad \text{and} \qquad
0\to N_k\to Y_k\oplus X_{k+1}\to M_k\to 0\,.
$$

$$
0\to Y_k\to M_k\oplus N_{k-1}\to X_k\to 0,\ k\geq 2 \qquad \text{and} \qquad
0\to Y_1\to M_1\oplus S\oplus A\to X_1\to 0\,.
$$

Symmetrically we obtain the following $\AR$-sequences.

$$
0\to X_k\to M_{k-1}\oplus N_k\to Y_k\to 0,\ k\geq 2 \qquad \text{and} \qquad
0\to X_1\to N_1\oplus B\to Y_1\to 0\,.
$$

$$
0\to A\to X_1\to B\to 0\qquad \text{and} \qquad
0\to B\to Y_1\to A\to 0\,.
$$

Note that $S$ is projective and $\Ext$-injective in the category of finitely generated $\CM$-modules, hence no
almost split sequence starts or ends in $S$. Also neither $C$ nor $D$ is a source or sink of an $\AR$-sequence
in this category.

Thus the $\AR$-quiver consists of 2 components: the top one consisting of modules $C$ and $D$, and the bottom one.
To explain why $C$ or $D$ are located as shown on Figure \ref{fig1}, note that the composite morphism
$Y_{k+1}\to N_k\to Y_k$ is just the inclusion of ideals $(y^{k+1},x)\subset (y^k, x)$. The inverse limit (i.e.
the intersection) along the coray of morphisms going through the $Y_k$ in the northwestern direction equals
$xS= D$. This is why $D$ is positioned at the end of this coray.

Similarly the composite map $M_k\to X_k\to M_{k-1}$ is the inclusion of ideals $(y^{k+1},xy)\subset (y^k, xy)$,
therefore $C= xyS$ is the inverse limit of the coray of morphisms going through the $M_k$ in the northwestern
direction.

Note (see \cite[Thm. 21.21]{Eis} for a general statement) that there is a duality of the category of finitely
generated $\CM$-modules over $S$ given by the functor $\Hom(-, S)$. It is easily checked that this duality
interchanges $X_k$ with $Y_k$, but fixes the $M_k, N_k$ and $A, B, C, D, S$.

\section{Some basics in model theory}

In this section we will recall few notions from the model theory of modules which are relevant to this paper.
We will use Prest's book \cite{Preb} as a main source for references.

Suppose that $U$ is a commutative ring. A \emph{positive primitive formula} $\phi(v)$ in one free variable $v$ is
an existential formula $\ex\, \ov w\, (\ov w A=  v \bar b)$ in the language of $U$-modules, where
$\ov w= (w_1, \dots, w_k)$ is a tuple of bound variables, $A$ is a $k\times l$ matrix over $U$, and $\ov b$ is
a column of height $l$ with entries in $U$. If $M$ is an $U$-module and $m\in M$ then we say that
\emph{$M$ satisfies $\phi(m)$}, written $M\ms \phi(m)$, if there exists a tuple $\ov m= (m_1, \dots, m_k)$ of
elements in $M$ such that $\ov m A= m\ov b$ holds.

For instance, if $r\in U$, then the \emph{divisibility formula} $r\mid v\doteq \ex\, w\, (wr= v)$ is satisfied by an
element $m\in M$ iff $m\in Mr$. Similarly, the \emph{annihilator formula} $vs=0$ holds on $m\in M$ iff $ms=0$.
Let $\phi$ and $\psi$ be pp-formulae. We say that $\phi$ \emph{implies} $\psi$, written $\phi\to \psi$, if, for
each pair $m\in M$, from $M\ms \phi(m)$ it follows that $M\ms \psi(m)$. For instance the divisibility formula
$rs\mid v$ implies $s\mid v$.

The implication relation defines a partial ordering $\leq$ on the set of pp-formulas. We say that the pp-formulae
$\phi$ and $\psi$ are \emph{equivalent} if both implications $\phi\to \psi$ and $\psi\to \phi$ hold. If we
identify equivalent pp-formulas, the resulting poset is a modular lattice $L$ called the \emph{lattice of pp-formulae}
(in one variable) over $U$. For instance the zero formula $v=0$ is the least element in this lattice, and the
trivial formula $v=v$ is the largest element of $L$. Further the meet in this lattice is the conjunction of
pp-formulae, and the join is the sum of pp-formulae: $(\phi+ \psi)(v)\doteq \ex\, w\, (\phi(w)\wg \psi(v-w))$.

If $\phi$ is a pp-formula and $M$ is an $U$-module, then $\phi(M)= \{m\in M\mid M\ms\phi(m)\}$ is the
\emph{definable submodule} of $M$. For instance the divisibility formula $r\mid v$ defines the submodule $Mr$,
and the annihilator formula $vs=0$ defines the submodule consisting of $m\in M$ such that $ms=0$.
This way we will obtain the \emph{lattice of pp-definable submodules of $M$}, $L(M)$, which is the factor of
the lattice $L$ of pp-formulae by certain congruence relation.

The correspondence $M\mapsto \phi(M)$ defines the finitely generated subfunctor of the functor $\Hom(U,-)$
from the category of finitely presented $U$-modules to the category of abelian groups, in fact the pp-formulae
can be introduced as such additive functors.

There are 'minimal' realizations of pp-formulae, called free realizations. Namely we say that the pointed module
$(M, m)$ is a \emph{free realization of a pp-formula $\phi$} if 1) $M\ms\phi(m)$ and 2) if $M\ms \psi(m)$ for 
some pp-formula $\psi$, then $\phi$ implies $\psi$. Though minimal realization is never unique, each pp-formula
has a free realization $(M,m)$ such that $M$ is a finitely presented module. For instance $(U,r)$ is a free
realization of the divisibility formula $r\mid x$, and $\ov 1\in U/sU$ freely realizes the annihilator formula
$vs=0$. We will often write $m\in M$ to denote the pointed module $(M,m)$. Further each finitely presented pointed
module $m\in M$ is a free realization of some pp-formula $\phi$ (which generates the pp-type of $m$ in $M$, - see
definitions below).

If $M$ is finitely presented, then the above condition 2) means that for each pointed module $(N,n)$ satisfying
$\phi$ there is a \emph{pointed morphism} from $(M,m)$ to $(N,n)$, i.e. a morphism $f: M\to N$ such that
$f(m)= n$.

All this can be relativized to any definable class of modules. For instance let us consider the class of
Cohen--Macaulay modules over $S$. Some nonequivalent pp-formulae in $L$ get identified when restricted to
$\CM$-modules, hence we obtain the \emph{lattice} of $\CM$-formulae, $L_{CM}$, as a proper factor of $L$. The
following arguments show that this definable class is \emph{covariantly finite} (see \cite{H-R}), i.e. admits
finitely generated free realizations.

Take any pp-formulae $\phi$ over $S$ and choose a free realization $m\in M$ for this formula, where $M$ is a
finitely generated $S$-module. Factor out the socle of $M$ obtaining $M'= M/\Soc(M)$, then kick off the socle
of $M'$ and so on. Since $M$ is noetherian, in finitely many steps we will obtain a $\CM$-module $\ov M$ over
$S$. Let $\ov m$ denote the canonical image of $m$ in this module and let a pp-formula $\phi_{CM}$ generate the
pp-type of $\ov m$ in $\ov M$. Directly from the definition we obtain the following.

\vspace{2mm}

1) $\ov \phi$ implies $\phi$ in the theory of all $S$-modules.

\vspace{2mm}

2) $\ov \phi$ is equivalent to $\phi$ in the theory of $\CM$-modules over $S$.

\vspace{2mm}

We will call the formulae of the form $\phi_{CM}$ the \emph{$\CM$-formulae} and use pointed modules
$\ov m\in \ov M$ to denote them. Thus the lattice $L_{\CM}$ can be described as follows. It consists of equivalence
classes of pointed finitely generated $\CM$-modules over $S$, where $(M,m)\leq (N,n)$ if there exists a pointed
morphism from the latter module to the former. For instance the formula $x^2\in S$ implies $x\in S$, but these
formulae are not equivalent.

Further, the sum of $\CM$-formulas $m\in M$ and $n\in N$ is given by the pointed module $(m,n)\in M\oplus N$, and
their conjunction is defined using the following pushout.

\vspace{2mm}

\begin{center}
$
\xymatrix@C=16pt@R=24pt{%
*+{(R,1)}\ar[r]\ar[d]&*+{(M,m)}\ar@{-->}[d]&\\
*+{(N,n)}\ar@{-->}[r]&*+{(K,k)}\ar[r]^{\pi}&*+{(\ov K, \ov k)}
}
$
\end{center}

\vspace{3mm}

\section{The first interval}\label{S-cut1}

In this section we will create the first cut along the category $\CM_{fg}$ of finitely generated $\CM$-modules
over $S$, i.e.\ we calculate a particular interval in the lattice of $\CM$-formulae, which later will be used as
a net to catch indecomposable pure injective modules. For this we use the following notion borrowed (with some
adjustments) from Ringel \cite[Sect. 3]{Rin80}, where he calculates the hammock functors.

Suppose that $(M,m)$ is a pointed module, where $M$ is an indecomposable finitely generated $\CM$-module over $S$.
Consider the set $P(M)$ of all pointed modules $(N,n)$ such that $N$ is an indecomposable finitely generated
$\CM$-module, and there is a pointed morphism from $(M,m)$ to $(N,n)$. In particular if $(M,m)$ realizes a
pp-formula $\phi$ freely and $(N,m)$ is a free realization of $\psi$, then $\psi$ implies $\phi$.

We introduce the partial order on $P$ by setting $(N,n)\geq (K,k)$ if there is a pointed morphism from the former
module to the latter. Finally we factor this collection of pointed modules by the corresponding equivalence
relation ($\leq$\, and\, $\geq$) to obtain the \emph{pattern} of $M$. Thus $P(M)$ is a partially ordered set with
the least element (corresponding to the zero morphism), and the largest element $1_M$; however it is not clear
that this poset is a lattice.

We keep in mind the interpretation of this poset in terms of pp-formulae: if $(M,m)$ is a free realization of
$\phi$, then $P(M)$ consists of formulae in the interval $[v=0, \phi]_{CM}$ with indecomposable free realizations,
i.e. of formulae which are \emph{$+$-irreducible}. Thus the pattern is the 'frame' of the corresponding interval
in the lattice of $\CM$-formulae.

Because (see Section \ref{S-fg}) we know indecomposable finitely generated $\CM$-modules, the calculation of
patterns is straightforward. Here is an important example.

\begin{prop}\label{1-patt}
The diagram on Figure \ref{fig2} shows the pattern of the pointed module $(S,x^2)$.
\end{prop}

\begin{figure}
$$
\vcenter{%
\xymatrix@C=12pt@R=14pt{%
&*+={\circ}\ar@{}+<20pt,0pt>*{_{(S,x^2)}}\ar[d]&\\
&*+={\circ}\ar@{}+<24pt,0pt>*{_{(X_1,mx^2)}}\ar[d]&\\
&*+={\circ}\ar@{}+<24pt,0pt>*{_{(N_1,mx^2)}}\ar[d]&\\
&*+={\circ}\ar@{}+<24pt,0pt>*{_{(X_2,mx^2)}}
\ar@{}+<0pt,-6pt>*{.}\ar@{}+<0pt,-12pt>*{.}\ar@{}+<0pt,-18pt>*{.}&\\
&*+={_{(D,x^3)}}\ar@{}+<0pt,8pt>*{}="s1"&\ar@{}+<-4pt,8pt>*{_{p_1}}="p1"\\
&*+={\circ}\ar@{}+<24pt,0pt>*{_{(Y_2,nx^2)}}\ar[d]
\ar@{}+<0pt,6pt>*{.}\ar@{}+<0pt,12pt>*{.}\ar@{}+<0pt,18pt>*{.}&\\
&*+={\circ}\ar@{}+<24pt,0pt>*{_{(N_1,nx^2)}}\ar[d]&\\
&*+={\circ}\ar@{}+<24pt,0pt>*{_{(Y_1,nx^2)}}\ar[dr]\ar[ld]&\\
*+={\circ}\ar@{}+<-20pt,0pt>*{_{(A,x^4)}}\ar[dr]&&*+={\circ}\ar@{}+<20pt,0pt>*{_{(S,x^3)}}\ar[dl]\\
&*+={\circ}\ar@{}+<24pt,-2pt>*{_{(X_1,mx^3)}}
\ar@{}+<0pt,-6pt>*{.}\ar@{}+<0pt,-12pt>*{.}\ar@{}+<0pt,-18pt>*{.}&\\
&*+={_{(D,x^4)}}\ar@{}+<0pt,8pt>*{}="s2"&\ar@{}+<-4pt,8pt>*{_{p_2}}="p2"\\
&*+={\circ}\ar@{}+<20pt,0pt>*{_{(N_,nx^3)}}\ar[d]\ar@{}+<0pt,6pt>*{.}\ar@{}+<0pt,12pt>*{.}\ar@{}+<0pt,18pt>*{.}&\\
&*+={\circ}\ar@{}+<20pt,2pt>*{_{(Y_1,nx^3)}}\ar[dl]\ar[dr]&\\
*+={\circ}\ar@{}+<-20pt,0pt>*{_{(A,x^5)}}\ar[dr]&&
*+={\circ}\ar@{}+<20pt,0pt>*{_{(S,x^4)}}\ar[dl]\\
&*+={\circ}\ar@{}+<24pt,-2pt>*{_{(X_1,mx^4)}}\ar@{}+<0pt,-6pt>*{.}\ar@{}+<0pt,-10pt>*{.}\ar@{}+<0pt,-14pt>*{.}\\
&_0\ar@{}+<0pt,8pt>*{}="s3"&\ar@{}+<-4pt,8pt>*{_q}="q"\\
\ar@{->}"p1"+<-6pt,0pt>;"s1"
\ar@{->}"p2"+<-6pt,0pt>;"s2"
\ar@{->}"q"+<-6pt,0pt>;"s3"
}}
$$
\caption{}\label{fig2}
\end{figure}

\begin{proof}
To calculate this pattern we should look at the image of $x^2$ with respect to morphisms $f$ from $S$ to
indecomposable finitely generated $\CM$-modules $M$.  Further we will identify morphism $f$ and $g$ such that
$f(x^2)$ and $g(x^2)$ are free realizations of equivalent pp-formulas. For example if $f(x^2)= 0$ then $f$
represent the zero element in the pattern. Also the irreducible morphism $S\to X_1$ which sends $1\mapsto m$ shows
that $(S, x^2)$ covers $(X_1, mx^2)$ in the pattern.

Note that there is an almost split morphism $f: X_1\to N_1\oplus B$, and the image of $mx^2$ via this map is
$(mx^2,0)$; hence $f$ is equivalent to the irreducible morphism $X_1\to N$ given by the first component of $f$.
It follows that $mx^2\in N_1$ is the largest formula strictly below $mx^2\in X_1$. The rest is by inspection.
\end{proof}

We have already noticed that Figure \ref{fig2} represents the frame of the interval below the formula $x^2\mid v$
in the lattice of $\CM$-formulae. We will recover the whole structure of this interval. First we will prove an
auxiliary lemma.

\begin{lemma}\label{dist}
The interval $[v=0, x^2\mid v]$ in the lattice $L_{CM}$ of all $\CM$-formulae over $S$ is distributive.
\end{lemma}
\begin{proof}
Recall that a module is said to be \emph{uniserial} if its lattice of submodules is a chain.

By \cite[Prop. 4.4]{P-P} it suffices to prove that the formula $x^2\mid v$, when evaluated on an indecomposable
finitely generated $\CM$-module $M$ over $S$, defines a submodule which is uniserial as a module over the 
endomorphism ring of $M$. For instance, there is nothing to prove for $M_k$, since $M_k x^2=0$.

The most difficult case is $N_k$, - let $V= \End(N_k)$. Note that there exists an endomorphism of this module
sending $m\mapsto n$,  $n\mapsto 0$; and there is an endomorphism such that $m\mapsto my-n$, $n\mapsto mx$.
Using this we conclude that nonzero endosubmodules of $N_k$ contained in $N_k x^2$ form the following chain:
$V mx^2\supset V nx^2\supset Vmx^3\supset \dots$. 
\end{proof}

Now we are in a position to describe the lattice of $\CM$-formulae in the above interval. Since every finitely
generated $\CM$-module is a direct sum of indecomposable modules, it follows that each formula in this interval
is equivalent to a sum $\phi_1+ \dots+ \phi_n$ of formulae with free realization given by modules on
Figure \ref{fig2}. Of course if $\phi_i$ implies $\phi_j$ in this sum, the former formula can be skipped.
Because most formulas on the diagram are comparable, it remains to consider only what happens if we add,
say, $x^4\in A$ to $x^3\in S$.

This situation admits a very general analysis. By using distributivity (Lemma \ref{dist}) and the fact that the
endomorphism ring of each indecomposable $\CM$-module on the diagram is local, as in \cite[Prop. 4.3]{Pun16} we
derive the following. If $\phi= \sum_i \phi_i$ and $\psi= \sum_j \psi_j$, where $\phi_i$ and $\psi_j$ are
pp-formulae on the diagram, then $\phi$ implies $\psi$ only by the trivial reason: if for each $i$ there exists $j$
such that $\phi_i\leq \psi_j$ in the pattern.

Thus to get the interval $[v=0, x^2\mid v]$ from the pattern we need only insert additional formulas
$(A,x^{k+1})+ (S,x^k)$, $k\geq 3$ above each summand and strictly below the formula $(Y_1, nx^{k-1})$. For
instance we will obtain the following diagram below $(Y_1, nx^2)$, where the top interval is simple. It
corresponds to the projection $Y_1\to A\oplus S$ onto the first two coordinates of the left almost split
morphism $Y_1\to A\oplus S\oplus M_1$.
\vspace{3mm}

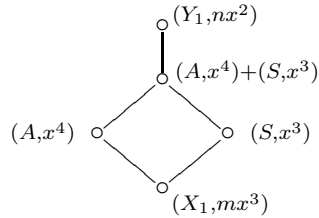
\begin{center}
$
\xymatrix@C=16pt@R=12pt{%
&*+={\circ}\ar@{-}[d]\ar@{}+<20pt,4pt>*{_{(Y_1,nx^2)}}&\\
&*+={\circ}\ar@{-}[ld]\ar@{-}[rd]\ar@{}+<32pt,3pt>*{_{(A,x^4)+(S,x^3)}}&\\
*+={\circ}\ar@{-}[rd]\ar@{}+<-20pt,0pt>*{_{(A,x^4)}}&&*+={\circ}\ar@{-}[ld]\ar@{}+<20pt,0pt>*{_{(S,x^3)}}\\
&*+={\circ}\ar@{}+<22pt,-6pt>*{_{(X_1,mx^3)}}&
}
$
\end{center}

\vspace{3mm}

It also follows that the conjunction of pp-formulas in this interval coincides with the intersection in the pattern.
Thus the pattern consists of $+$-irreducible elements of the interval, and is a lattice by itself, but the sums
in these lattices may be different.

\section{The Ziegler spectrum}\label{S-zieg}

In this section we classify points of the Cohen--Macaulay part of the Ziegler spectrum of $S$.

Let $U$ be a commutative ring. An inclusion of $U$-modules $M\seq N$ is said to be \emph{pure} if, for each
pp-formula $\phi$ and an element $m\in M$, from $N\ms \phi(m)$ it follows that $M\ms \phi(m)$. For instance
each injective module is pure in any module containing it. This definition is obviously extended to embedding
of modules.

A module $M$ is said to be \emph{pure injective} if it is injective with respect to pure embeddings. For instance
each injective module is pure injective, so as (see \cite{Rin00}) any module which is linearly compact over its
endomorphism ring. Since $U$ is commutative, each linearly compact module is pure injective.

The \emph{Ziegler spectrum of $U$}, $\Zg_U$, is a topological space whose points are indecomposable pure injective
modules. The topology on this space is given by \emph{basic open sets} $(\phi/\psi)$, where $\phi$ and $\psi$ are
pp-formulae. Here $(\phi/\psi)$ consists of points $M\in \Zg_S$ such that $\phi(M)$ is not contained in $\psi(M)$.
For instance if $r\in U$, then the pair $vr=0$ over $v=0$ defines the open set consisting of indecomposable pure
injective modules with $r$-torsion. Each basic open set is compact, in particular $\Zg_U$ is a quasi-compact
space which is usually non-Hausdorff.

Definable classes (categories) $\mD$ of $U$-modules can be identified with closed subsets of the Ziegler spectrum
whose points are indecomposable pure injective modules in $\mD$, and the topology is induced from the whole
spectrum. This way we obtain the Ziegler spectrum of this category. For instance, by considering the category
of Cohen--Macaulay $S$-modules we will obtain the \emph{Cohen-Macaulay part of the Ziegler spectrum} of $S$,
written $\ZCM_S$. To investigate this topological space is the main objective of this paper.

First we create a sufficient supply of points in this space.

\begin{lemma}
Each indecomposable finitely generated $\CM$-module over $S$ is pure injective, hence is the point of $\ZCM_S$.
Further the same holds true for $Q_R$, $\wt R$, and the generic modules $G_y= F((y))$, $G_x= F((x))$.
\end{lemma}
\begin{proof}
It is not difficult to check that each module in the above list is linearly compact, which implies pure injectivity.
For instance for $R$-modules in this list, hence for $B, C, M_k$ and $Q_R, \wt R, G_y$, this follows from
\cite{Pun16}.
\end{proof}

We need to add one more module to this list. If we identify the module $X_k$ with the submodule of $Q_S$ generated
by $1$ and $n_k= xy(x+y)^{-k}$, $k=1, \dots$, then these modules form the natural ray of inclusions, via
$n_k\mapsto n_{k+1} y$, going in the northeastern direction on Figure \ref{fig1}. The direct limit along this ray
will give us the $\CM$-module $N$, which is the submodule of $Q_S$ generated by $1$ and the $n_k$.

\vspace{4mm}

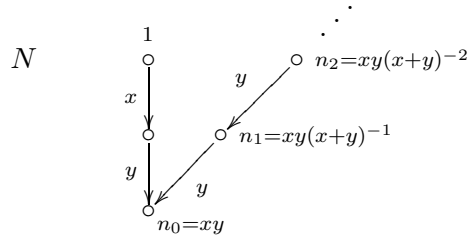
\begin{center}
$
N\hspace{10mm}
\xymatrix@C=20pt@R=20pt{%
*+={\circ}\ar[d]_x\ar@{}+<0pt,10pt>*{_1}&&*+={\circ}\ar[dl]_y\ar@{}+<36pt,0pt>*{_{n_2= xy(x+y)^{-2}}}
\ar@{}+<10pt,10pt>*{.}\ar@{}+<15pt,15pt>*{.}\ar@{}+<20pt,20pt>*{.}\\
*+={\circ}\ar[d]_y&*+={\circ}\ar[dl]^y\ar@{}+<36pt,0pt>*{_{n_1= xy(x+y)^{-1}}}&\\
*+={\circ}\ar@{}+<16pt,-6pt>*{_{n_0= xy}}&&
}
$
\end{center}

\vspace{3mm}

\begin{lemma}\label{n}
$N$ is indecomposable and pure injective as an $S$-module, hence is a point in $\ZCM_S$.
\end{lemma}
\begin{proof}
Let $N'$ be the submodule of $N$ generated by the $n_k$, hence $N'= xyQ_S$. Clearly $N'$ is isomorphic to $G_y$,
hence uniserial. Further $N'$ is a linearly compact module, so as the factor module $N/N'\cong D=xS$. It follow
that $N$ is linearly compact, hence pure injective.

It is straightforward to check that the endomorphism of $N_S$ is $N$ itself, which is a local ring. Thus $N_S$ is
an indecomposable module, hence is a point in $\ZCM_S$.
\end{proof}

The following property is similar to what is known for finite dimensional algebras (see \cite[Cor. 5.3.37]{Preb}).

\begin{lemma}\label{dense}
Finitely generated points are dense in the $\CM$-part of the Ziegler spectrum of $S$.
\end{lemma}
\begin{proof}
Clearly each $\CM$-module over $S$ is a direct limit of its finitely generated submodules, which are
Cohen--Macaulay, hence the result follows.
\end{proof}

To approach the classification of points in $\ZCM_S$ we need more explanations. Let $M$ be a module over a
commutative ring $U$, and let $m$ be a (usually nonzero) element of $M$. The \emph{positive primitive type}
$p= pp_M(m)$ consists of all pp-formulae which are satisfied by $m$ in $M$. Clearly this set form a \emph{filter},
i.e.\ it is closed with respect to finite conjunctions and implications. The converse is also true: for each
filter $p$ in the lattice of pp-formulae there exists a module $M$ and an element $m\in M$ such that $p=pp_M(m)$.
There always exists a minimal pure injective module $M$ realizing $p$, which is unique up to an isomorphism over
the realization. This module is called the \emph{pure injective envelope of $p$}, written $\PE(p)$.

We say that a pp-type $p$ is \emph{indecomposable} if its pure injective envelope is an indecomposable module.
There is an important syntactical criterion which recognizes the indecomposability of a type. By $p^-$ we will
denote the complement of $p$, i.e.\ the collection of pp-formulae which are not in $p$. For an indecomposable
pp-type this collections is almost a cofilter.

\begin{fact}\label{indec-crit}(see \cite[Thm. 4.4]{Zieg})
A pp-type $p$ is indecomposable iff for each $\phi_1, \phi_2\in p^-$ there exists $\phi\in p$ such that
$(\phi_1\wg \phi)+ (\phi_2\wg \phi)\in p^-$.
\end{fact}

Suppose that $\phi$ and $\psi$ are pp-formulae such that $\phi$ does not imply $\psi$, hence we have the
nontrivial interval $[\psi\wg \phi, \phi]$ in the lattice of pp-formulae. Suppose that $p$ is a pp-type such
that $\phi\in p$ and $\psi\in p^-$. Then $p$ defines, by restriction, a \emph{cut} on this interval, whose upper
part (i.e. the formulae in $p$ from this interval) is a filter. The following fact shows that (like analytic
functions) the pure injective envelope of $p$ can be recovered form any cut, whatever small it be.

\begin{fact}\label{recov}(see \cite[Cor. 4.16]{Zieg})
Let $p$ be an indecomposable pp-type with the pure injective envelope $M$. If $\phi\in p$ and $\psi$ in $p^-$ are
pp-formulae, then the isomorphism type of $M$ is uniquely determined by the cut defined by $p$ on the interval
$[\phi\wg \psi; \phi]$.

Further the open sets $(\phi'/\psi')$, where $\psi\leq \psi'< \phi'\leq \phi$ and $\phi'\in p$, $\psi'\in p^-$,
form a basis of open sets for $M$ in Ziegler topology.
\end{fact}

For instance, the irreducibility criterion from Fact \ref{indec-crit} can be used locally. Namely suppose that
$p$ is predefined by a filter in a given interval, and is indecomposable in there, i.e. satisfies the above
conditions with the $\phi_i$ and $\phi$ within the interval. Then $p$ can be extended to an indecomposable pp-type,
and all such extensions (whatever different) lead to isomorphic indecomposable pure injective modules.

All this relativizes to any definable class of modules, with obvious changes. From this we derive the general
approach to a classification of $\CM$-modules over $S$. First choose a 'nice' interval whose lattice structure
is known. Then using Ziegler's criterion describe indecomposable pp-types in this interval (or rather cuts of
such types). Finally find indecomposable pure injective modules which realize these pp-types locally. This will
complete the classification.

Now we are in a position to classify points of $\ZCM_S$.

\begin{theorem}\label{points}
Let $M$ be an indecomposable pure injective Cohen-Macaulay module over $S$. Then either it is finitely generated,
or $M$ is isomorphic to one of modules $Q_R$, $\wt R$, $N$, $G_y= F((y))$ and $G_x= F((x))$.
\end{theorem}
\begin{proof}
If $M$ is annihilated by $x^2$, then it is pure injective indecomposable $\CM$-module over $R$. It follows from
\cite{Pun16} that $M$ is one of the finitely generated modules $B, C$, $M_k$; or one of the modules $Q_R, \wt R$
or $G_y$.

Otherwise there exists a nonzero $m\in Mx^2$. Thus if $p= pp_M(m)$ then $x^2\mid v\in p$ and $v=0\in p^-$, hence
this type (therefore the module) is visible within the interval on Figure \ref{fig2}. Furthermore to apply the
Ziegler criterion for indecomposability (see Fact \ref{indec-crit}) note that this interval is a distributive
lattice. Thus a pp-type $p$ within this interval is indecomposable iff $p^-$ is a \emph{cofilter}, i.e.
$\phi_1+ \phi_2\in p^-$ for all pp-formulas $\phi_i$ in the negative part of $p$.

It easily follows that any cut on the frame gives an indecomposable pp-type. The \emph{principal cut} whose upper
part is generated by one pp-formula is realized in the corresponding indecomposable finitely generated
$\CM$-module. For instance the cut whose positive part is generated by $(D,x^3)$ corresponds to the module $D$.

What remains is to identify the modules which are produced by the infinitely generated cuts $p_i$, $i\geq 1$, and
by $q$ (see Figure \ref{fig2}). The latter case is easy to resolve. Namely the element $1\in G_x= F((x))$ is
divisible by any power of $x$, hence defines the required cut. Thus $\PE(q)\cong G_x$.

Now it is easily checked that the element $x^2\in N$ realizes the pp-type whose restriction to our interval is
$p_1$; furthermore $x^k$, $k> 2$ realizes in $N$ the pp-type, whose restriction coincides with $p_{k-1}$. Thus
we have filled in all vacant positions for points.
\end{proof}

Recall that a pair $(\phi/\psi)$ of pp-formulae is said to be \emph{minimal} in a theory $T$ if $\psi< \phi$, and
there is no pp-formula strictly between $\psi$ and $\phi$ in the lattice of pp-formulas relativized to $T$.

From Fact \ref{recov} it is easy to construct a basis of open sets for each point visible on Figure \ref{fig2}.
For instance we immediately obtain the following.

\begin{lemma}\label{isol}
Each finitely generated point which occurs on Figure \ref{fig2}, but $D$, is isolated by a minimal pair in the
theory $T_{CM}$ of $\CM$-modules over $S$.
\end{lemma}

For example $X_1$ is isolated by the pair $mx^2\in X_1$ over $mx^2\in N_1$, which corresponds to the irreducible
map $X_1\to N_1$.

Now we proceed with calculating the Cantor--Bendixson rank of points. The \emph{Cantor--Bendixson analysis} of a
closed subset $T$ of the Ziegler spectrum runs as follows. The first derivative $T'$ is obtained from $T$ by deleting
isolated points. Those points are assigned the $\CB$-rank 0. Further, $T''$ is obtained from $T'$ by deleting its
isolated points, which are assigned the $\CB$-rank 1, and so on. If $T'''$ is an empty set, then we say that the
$\CB$-rank of $T$ equals $2$.

There is a parallel analysis of the corresponding definable category which leads to the notion of the
\emph{Krull--Gabriel dimension} (see \cite[Ch. 7]{Kra}). Suppose that the \emph{isolation condition} holds true,
i.e. (see \cite[Sect. 5.3.2]{Preb}) on each stage of the $\CB$-analysis, each point is isolated by the minimal pair.
Then both procedures give the same value. This means that that lattice of pp-formulae of the next derivative of $T$
is obtained from the previous one by collapsing intervals of finite length. The corresponding dimension of the
lattice is called the \emph{$m$-dimension} in \cite[Sect. 7.7.2]{Preb}, and coincides with the Krull--Gabriel
dimension of the definable category.

We apply this analysis to the theory $T_{CM}$ of $\CM$-modules over $S$. We will see below that each isolated
point in $T_{CM}$ is isolated by the minimal pair. It follows that the lattice of pp-formulae $L_1$ of $T'_{CM}$
is obtained from the lattice $L$ of $\CM$-formulae over $S$ by collapsing intervals of finite length.

All this is applicable to any given interval, say, to the one given by Figure \ref{fig2}. We conclude that the 
interval $[v=0, x^2\mid v]$ in $T'_{CM}$ is the following chain of order type $1+ \om^*$.

\vspace{2mm}

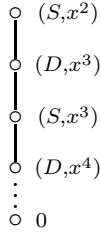
$$
\vcenter{%
\xymatrix@C=12pt@R=14pt{%
*+={\circ}\ar@{}+<20pt,0pt>*{_{(S,x^2)}}\ar@{-}[d]\\
*+={\circ}\ar@{}+<20pt,0pt>*{_{(D,x^3)}}\ar@{-}[d]\\
*+={\circ}\ar@{}+<20pt,0pt>*{_{(S,x^3)}}\ar@{-}[d]\\
*+={\circ}\ar@{}+<20pt,0pt>*{_{(D,x^4)}}\ar@{}+<0pt,-6pt>*{.}\ar@{}+<0pt,-10pt>*{.}\ar@{}+<0pt,-14pt>*{.}\\
*+={\circ}\ar@{}+<10pt,0pt>*{_0}
}}
$$

\vspace{3mm}

Now we are in a position to calculate the rank of $D$.

\begin{lemma}\label{d}
$D$ is isolated in the first derivative $T'_{CM}$ by the minimal pair $(D,x^3)$ over $(S,x^3)$, in particular
the Cantor--Bendixson rank of $D$ equals 1.
\end{lemma}
\begin{proof}
From Figure \ref{fig2} we conclude that the basis of open sets for $D$ in the Ziegler topology is given by
pairs of formulae $x^3\in D$ over $nx^k\in Y_k$, $k= 1, 2, \dots$, in particular $D$ is not isolated. It remains
to look at the above diagram.
\end{proof}

To include into consideration the module $N$ we need one more definition. A point $M$ in a closed set $T$ is said
to be \emph{neg-isolated} (relative to $T$), if there is a nonzero $m\in M$ such that the negative part of the
pp-type of $m$ contains the largest formula. By \cite[Prop. 5.3.46]{Preb} this notion does not depend on the
choice of $m$.

\begin{lemma}\label{n-isol}
$N$ is isolated in $T'_{CM}$ by the minimal pair $x^2\in S$ over $x^3\in D$, in particular $N$ has $\CB$-rank 1.
Further $N$ is neg-isolated in $T_{CM}$.
\end{lemma}
\begin{proof}
The basis of open sets for $N$ is given by pairs $mx^k\in X_k$, $k=1, 2, \dots$ over $x^3\in D$, in particular
this point is neg-isolated and not isolated. It remains to look at the above diagram.
\end{proof}

Thus the only remaining point in the above interval is $G_x= \PE(q)$.

\begin{lemma}\label{gx}
$G_x$ has Cantor--Bendixson rank $2$.
\end{lemma}
\begin{proof}
From Figure \ref{fig2} we see that the basis of open sets for $G_x$ is given by pairs $x^k\in S$, $k= 1, 2, \dots$
over $0$, hence $q$ is a pp-type maximal over $0$ (usually called \emph{critical}). It easily follows that
$G_x$ is not isolated on levels zero and one of the $\CB$-analysis.

By applying the $m$-dimension analysis to the interval from the above diagram, we obtain the two-point interval.
It follows that $G_x$ is the only point in the interval $x^2\in S$ over $0$ in $T''_{CM}$.
\end{proof}

\section{The second interval}\label{S-cut2}

To complete the description of topology of $\ZCM_S$ we need to calculate the bases of open sets for points from
this space which are not visible on Figure \ref{fig2}, hence for indecomposable pure injective $\CM$-modules
over $R$. Here is a complete list of these points.

\vspace{2mm}

\textbf{1)} $B$ and $M_k$, $k\geq 1$; \quad \textbf{2)} $C$ and $\wt R$; \quad \textbf{3)} $G_y= F((y))$ and $Q_R$.

\vspace{2mm}

Furthermore (from \cite{Pun16}) we know that the $\CB$-ranks of these points in the relative topology of $\ZCM_R$
are $0, 1$ and $2$, but we have recalculate these ranks in the ambient space $\ZCM_S$. For this we construct
another interval in the lattice $L$ of $\CM$-formulae, which captures points in question.

\begin{prop}\label{cut2}
The diagram on Figure \ref{fig3} represents the pattern of the pointed module $(C, xy)$.
\end{prop}

\begin{figure}[b]
$$
\vcenter{%
\xymatrix@C=16pt@R=14pt{%
&&&&*+={\circ}\ar@{}+<16pt,4pt>*{_{xy\in C}}\ar[dr]&&&\\
&&&&&*+={\circ}\ar@{}+<16pt,4pt>*{_{xy\in D}}\ar[dr]&&&\\
&&*+={\circ}\ar@{}+<-16pt,4pt>*{_{n\in M_1}}\ar[dr]\ar[dl]\ar@{}+<8pt,8pt>*{.}\ar@{}+<14pt,14pt>*{.}\ar@{}+<20pt,20pt>*{.}
&&&&*+={\circ}\ar@{}+<16pt,4pt>*{_{xy^2\in C}}\ar[dr]\\
&*+={\circ}\ar@{}+<-16pt,4pt>*{_{n\in X_1}}\ar[dr]\ar[dl]&&
*+={\circ}\ar@{}+<-18pt,0pt>*{_{ny\in Y_2}}\ar[dl]\ar[dr]\ar@{}+<8pt,8pt>*{.}\ar@{}+<14pt,14pt>*{.}\ar@{}+<20pt,20pt>*{.}
&&&&*+={\circ}\ar@{}+<17pt,4pt>*{_{xy^2\in D}}\ar@{}+<8pt,-8pt>*{.}\ar@{}+<14pt,-14pt>*{.}\ar@{}+<20pt,-20pt>*{.}&\\
*+={\circ}\ar@{}+<-16pt,4pt>*{_{xy\in B}}\ar[dr]&&*+={\circ}\ar@{}+<-18pt,0pt>*{_{ny\in N_1}}\ar[dr]\ar[dl]&&
*+={\circ}\ar@{}+<-18pt,0pt>*{_{ny\in M_2}}\ar[dr]\ar[dl]\ar@{}+<8pt,8pt>*{.}\ar@{}+<14pt,14pt>*{.}\ar@{}+<20pt,20pt>*{.}
&&&&&\\
&*+={\circ}\ar@{}+<-18pt,0pt>*{_{ny\in Y_1}}\ar[d]\ar[dr]&&*+={\circ}\ar@{}+<-18pt,0pt>*{_{ny\in X_2}}\ar[dr]\ar[dl]&
&*+={\circ}\ar@{}+<-18pt,0pt>*{_{ny^2\in Y_3}}\ar[dl]\ar@{}+<8pt,8pt>*{.}\ar@{}+<14pt,14pt>*{.}\ar@{}+<20pt,20pt>*{.}
\ar@{}+<8pt,-8pt>*{.}\ar@{}+<14pt,-14pt>*{.}\ar@{}+<20pt,-20pt>*{.}
&&&&\\
&*+={\circ}\ar@{}+<-18pt,0pt>*{_{xy\in S}}\ar[d]&*+={\circ}\ar@{}+<18pt,0pt>*{_{ny\in M_1}}\ar[dr]\ar[dl]&&
*+={\circ}\ar@{}+<20pt,0pt>*{_{ny^2\in N_2}}\ar[dl]\ar@{}+<8pt,-8pt>*{.}\ar@{}+<14pt,-14pt>*{.}\ar@{}+<20pt,-20pt>*{.}
&&&\\
&*+={\circ}\ar@{}+<-18pt,0pt>*{_{xy\in X_1}}\ar[dr]\ar[dl]&&
*+={\circ}\ar@{}+<18pt,0pt>*{_{ny^2\in Y_2}}\ar[dl]\ar@{}+<8pt,-8pt>*{.}\ar@{}+<14pt,-14pt>*{.}\ar@{}+<20pt,-20pt>*{.}
&&&&&\\
*+={\circ}\ar@{}+<-18pt,0pt>*{_{xy^2\in B}}\ar[dr]&&
*+={\circ}\ar@{}+<20pt,0pt>*{_{ny^2\in N_1}}\ar[dl]\ar@{}+<8pt,-8pt>*{.}\ar@{}+<14pt,-14pt>*{.}\ar@{}+<20pt,-20pt>*{.}
&&&&&&\\
&*+={\circ}\ar@{}+<18pt,0pt>*{_{ny\in Y_2}}\ar@{}+<8pt,-8pt>*{.}\ar@{}+<14pt,-14pt>*{.}\ar@{}+<20pt,-20pt>*{.}&&&&
&&&\\
&&&&&&&&&&\\
&&&&&
}}
$$
\caption{}\label{fig3}
\end{figure}

\begin{proof}
By inspection. A very similar proof (with a similar resulting diagram) is given in \cite[Prop. 4.3]{Pun16}.
\end{proof}

Note that the top pair $xy\in C$ is a free realization of the formula $vx=0$. It follows that the interval
$[v=0, vx=0]$ in $L$ is generated by pp-formulae on the figure with respect to free sums, for instance it is
distributive. This diagram allows us to calculate the ranks of remaining points.

\begin{lemma}\label{mk}
Each finitely generated point $M_k$ is isolated in $\ZCM_S$ by a minimal pair, and the same holds true for $B$.
\end{lemma}
\begin{proof}
Directly from the diagram, because the $M_k$ and $B$ are sources of almost split sequences in the category
of finitely generated $\CM$-modules over $S$.

For instance, the module $M_1$ is isolated by the pair $n\in M_1$ over $(X_1,n)+ (Y_2,ny)$ which corresponds
to the left almost split map $M_1\to X_1\oplus Y_2$.
\end{proof}

It also follows that none of the remaining point is isolated in $T_{CM}= \ZCM_S$, in particular each isolated
point is isolated by a minimal pair. For instance, the lattice $L_1$ of pp-formulae in $T'_{CM}$ coincides with
the lattice obtained from the lattice $L$ of $\CM$-formulae by collapsing intervals of finite length. In
particular this holds true for each interval.

\begin{lemma}\label{2prime}
The interval $xy\in C$ over zero in $L_1$ is the following chain of order type $1+ \Z+ \om^*$.

\vspace{2mm}

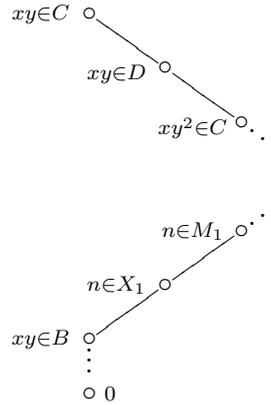
\begin{center}
$
\xymatrix@C=20pt@R=12pt{%
*+={\circ}\ar@{}+<-18pt,0pt>*{_{xy\in C}}\ar@{-}[dr]&&\\
&*+={\circ}\ar@{}+<-18pt,-2pt>*{_{xy\in D}}\ar@{-}[dr]&\\
&&*+={\circ}\ar@{}+<-18pt,-2pt>*{_{xy^2\in C}}\ar@{}+<4pt,-3pt>*{.}\ar@{}+<8pt,-6pt>*{.}\ar@{}+<12pt,-9pt>*{.}\\
&&\\
&&*+={\circ}\ar@{}+<-18pt,0pt>*{_{n\in M_1}}\ar@{-}[dl]\ar@{}+<4pt,3pt>*{.}\ar@{}+<8pt,6pt>*{.}\ar@{}+<12pt,9pt>*{.}\\
&*+={\circ}\ar@{}+<-18pt,0pt>*{_{n\in X_1}}\ar@{-}[dl]&\\
*+={\circ}\ar@{}+<-18pt,0pt>*{_{xy\in B}}\ar@{}+<0pt,-4pt>*{.}\ar@{}+<0pt,-8pt>*{.}\ar@{}+<0pt,-12pt>*{.}&&\\
*+={\circ}\ar@{}+<8pt,0pt>*{_0}&&
}
$
\end{center}

\end{lemma}
\begin{proof}
By direct calculations. For instance the interval $n\in X_1$ over $ny\in N_1$ includes only one additional formula
$(B,xy)+ (N_1,ny)$, hence has length 2, and therefore becomes trivial in $L_1$. Similarly each consecutive
interval along the southeastern ray $(X_1,n)\to (N_1,ny)\to (X_2, ny)\to \dots$ has finite length in $L$, hence
gets collapsed in $L_1$.

On the other hand, the interval $(X_1,n)$ over $(B,xy)$ contains the descending chain $(B,xy)+ (X_{k+1}, ny^k)$,
$k= 1, 2, \dots$, hence produces a nontrivial (in fact simple) interval in $L_1$.
\end{proof}

Recall that $C$ and $\wt R$ have $\CB$-rank 1 in $\ZCM_R$, therefore their ranks in $\ZCM_S$ cannot be smaller.
To our surprise these ranks are preserved in the ambient space.

\begin{lemma}\label{rk-c}
$C= xyS$ is isolated in $T'_{CM}$ by the minimal pair $(C,xy)$ over $(D,xy)$, in particular $\CB(C)= 1$.
\end{lemma}
\begin{proof}
From Figure \ref{fig3} we see that the basis of open sets for $C$ is given by pairs $xy\in C$ over
$(D,xy)+ (M_k,n)$. By evaluating it is easily seen that any such pair is open on $M_{k+1}$, therefore $C$
is not isolated.

On the other hand in $L_1$ the interval $(C,xy)$ over $(D,xy)$ is simple, hence isolates $C$.
\end{proof}

Now we proceed with $\wt R$.

\begin{lemma}\label{tr-rank}
$\wt R$ is isolated in $T'_{CM}$ by the minimal pair $(M_1,m)$ over $(X_1,n)$, in particular $\CB(\wt R)= 1$.
Further $\wt R$ is neg-isolated.
\end{lemma}
\begin{proof}
From Figure \ref{fig3} we see that the basis for the pointed module $(\wt R, xy)$ is given by the ray heading
from $(M_1,n)$ southeastern, hence by the pairs $ny^{k-1}\in M_k$ over $n\in X_1$. In particular the latter
formula neg-isolates the pp-type of $xy$ in $\wt R$. Since each such pair opens $M_{k+1}$, this point is not
isolated. It remains to look at the above diagram.
\end{proof}

It remains to consider the points $Q_R$ and $G_y$ whose $\CB$-rank is at least 2. From the above analysis it
follows that the lattice of pp-formulae of $T''_{CM}$ coincides with the lattice $L_2$ obtained from $L_1$ by
collapsing the finite length intervals. From the above diagram we conclude the the interval $xy\in C$ over
zero in $L_2$ is the following chain of length 2.

\vspace{2mm}

\begin{center}
$
\xymatrix@C=20pt@R=16pt{%
*+={\circ}\ar@{}+<18pt,0pt>*{_{xy\in C}}\ar@{-}[d]\\
*+={\circ}\ar@{}+<18pt,0pt>*{_{xy\in B}}\ar@{-}[d]\\
*+={\circ}\ar@{}+<8pt,0pt>*{_0}
}
$
\end{center}

\vspace{3mm}

The simple intervals in this chain will catch the remaining points of $\ZCM_S$.

\begin{lemma}\label{qr}
$Q_R$ is isolated in $T''_{CM}$ by the minimal pair $(B,xy)$ over $0$, hence $\CB(Q_R)=2$. Further, $G_y$ is
isolated in $T''_{CM}$ by the minimal pair $(C,xy)$ over $(B,xy)$, hence $\CB(G_y)= 2$.
\end{lemma}

Now we conclude on the global value of the $\CB$-rank.

\begin{theorem}\label{cb-rk}
The Cantor--Bendixson rank of the Cohen--Macaulay part of the Ziegler spectrum of $S$ equals 2. Further
the same holds true for the Krull--Gabriel dimension of the (definable) category of Cohen--Macaulay $S$-modules.
\end{theorem}
\begin{proof}
The first part follows from classification of points in $\ZCM_S$ (see Theorem \ref{points}), and the values
of their ranks calculated in this and the previous section. Further at each stage of the $\CB$-analysis, each
isolated point is isolated by the minimal pair. It follows that the Krull--Gabriel dimension of the category of
$\CM$-modules coincides with the $m$-dimension of the lattice $L$, and with the $\CB$-rank of the space $\ZCM_S$,
which have been already given the value.
\end{proof}

Note that the basic open set $(vx=0/v=0)$ contains all points in $\ZCM_S$, but $A$ and $G_x$.

\section{Auslander--Reiten sequences}

It follows from the above analysis that almost split sequences we mentioned in Section \ref{S-fg} preserve their
defining properties in the category of all (finitely generated or not) $\CM$-modules.

In this section we will produce few more $\AR$-sequences that will be used in the next section to glue a surface
from the category of finitely generated $\CM$-modules over $S$. Though no direct reference is possible, the
ideology goes back to Auslander \cite{Aus} who constructed almost split sequences in the category of all
$U$-modules (over any ring $U$) targeting a finitely presented module with a local endomorphism ring. As
follows from Herzog's interpretation of this result (see \cite{Her}), the source of this $\AR$-sequence is an
indecomposable pure injective module which is neg-isolated.

This was how we guessed almost split sequences in the category of $\CM$-modules over $S$ with infinitely generated
terms. Namely there are few indecomposable infinitely generated pure injective $\CM$-modules which are neg-isolated, 
hence few candidates for sources of $\AR$-sequences. It was difficult to guess, but it is surprisingly easy to 
check the required properties.

Recall that $N$ is identified with the submodule of $Q_S$ generated by $1$ and the $n_k= xy(x+y)^{-k}$. Further
$N'$ is the submodule of $N$ generated by the $n_k$, hence $N'= xyQ_S$. Then the factor $N/N'$ is isomorphic to
$D= xS$ via the map which sends $\ov 1$ to $x$. We will also identify $\wt R$ with the submodule $yS+ N'$ of $Q_S$.
In particular $\wt R x= Sxy$ and $\wt R/\wt Rx\cong yS/xyS$, which is isomorphic to $S/xS\cong C= xyS$.

With this identification in mind we will produce the first $\AR$-sequence with infinitely generated terms.

\begin{lemma}\label{ar1}
The following is an almost split sequence in the category of $\CM$-modules over $S$.

$$
0\to \wt R\xr{f= (\ih_1, \pi_1)} N\oplus C\xr{u= (\pi'_1, -\ih'_1)} D\to 0\,. \label{1}
$$

Here $\ih_1$ denotes the inclusion $\wt R= yS+ N'\subset N$, and $\pi_1$ is the above epimorphism
$\wt R\to \wt R/\wt Rx\cong C$ which sends $y\in\wt R$ to $xy\in C$, hence is given by multiplication by $x$.

Further $\pi'_1$ is a composition of the epimorphism $N\to N/N'$ with the isomorphism $N/N'\cong D= xS$ such
that $1\mapsto x$, hence is given by multiplication by $x$; and $\ih'_1$ is the inclusion $C\subset D$.
\end{lemma}
\begin{proof}
It is not difficult to check that this sequence is exact. For instance $uf(y)= u(y,xy)= xy-xy=0$. Further
the cokernel of $f$ is obtained by identifying $y\in N$ with $-xy\in C$, hence kicks off $N'$, and the resulting
module is isomorphic to $N/N'= D$.

Since the endomorphism ring of the target $D$ is local, it follows from \cite[Prop. 4.4]{Aus} that it
suffices to check that $f$ is left almost split. Suppose that $K$ is a $\CM$-module over $S$, and
$g: \wt R\to K$ is a morphism which is not a split monomorphism. We need to find $h: N\oplus C\to K$
completing the following diagram.

\begin{center}
$
\xymatrix@C=20pt@R=14pt{%
*+{\wt R}\ar[r]^(.4)f\ar[d]_g&*+{N\oplus C}\ar@{-->}[dl]^h\\
*+{K}&
}
$
\end{center}

\vspace{2mm}

Note that $f(xy)= (xy, 0)$ belongs to $N$.

Assume first that $g(xy)\neq 0$ and let $q$ be the pp-type of $g(xy)$ in $K$. Since $g$ is not a split monomorphism,
it increases the pp-type $p$ of $xy\in \wt R$, which (see Figure \ref{fig3}) is neg-isolated  by the formula
$n\in X_1$, hence this formula is in $q$. Since $q$ is a filter, by taking conjunctions we obtain that each formula
in the northeastern ray starting from $n\in X_1$ belongs to $p$. We conclude that $q$ includes each
formula in the pp-type of $xy\in N$. It follows that there exist a morphism $h: N\to K$ (considered as a morphism
$N\oplus C\to K$) such that $g(xy)= hf(xy)$.

Thus by considering $g'= g- hf: \wt R\to N\oplus K$ we may assume that $g(xy)=0$. Since $K$ is Cohen--Macaulay,
it follows that $g(xyQ_S)= 0$, therefore $g$ factors through $\wt R/xyQ_S\cong C$, hence through $\pi_1$.
\end{proof}

By swapping $C$ with $D$, and $\wt R$ with $N$, we obtain a 'dual' almost split sequence.

\begin{lemma}\label{ar2}
The following is an almost split sequence in the category of $\CM$-modules over $S$.

$$
0\to N\xr{f= (\ih_2, \pi'_1)} \wt R\oplus D\xr{u= (\pi_1, -\ih'_2)} C\to 0\,. \label{2}
$$

Here $\ih_2: N\to \wt R$ and $\ih'_2: D\to C$ are given by multiplication by $y$. The epimorphisms
$\pi_1$ and $\pi'_1$ are defines in Lemma \ref{ar1}.
\end{lemma}
\begin{proof}
We will give only the beginning of the proof. The remaining part is similar to the previous lemma.

Because $uf(1)= u(y,x)= xy-xy=0$, the composite map is zero. Further, because $y\in\wt R$ is identified with
$-x\in D$, we conclude that $xy^2=0$ in the cokernel. Because this element equals $xy(x+y)$, we derive $xy=0$,
hence the submodule $xyQ_S$ of $\wt R$ is annihilated.
\end{proof}

By projecting the above almost split sequences we produce the following irreducible maps.

\begin{cor}\label{irr}
The following morphisms are irreducible in the definable category of $\CM$-modules over $S$.

1) The epimorphisms $\wt R\xr{x} C$, $N\xr{x} D$, $D\xr{y} C$ and $N\xr{y}\wt R$.

2) The monomorphisms $\wt R\subset N$ and $C\subset D$.
\end{cor}

\section{The Ringel quilt}\label{S-quilt}

In this section we will add indecomposable infinitely generated pure injective $\CM$-modules over $S$ to the
$\AR$-quiver of the category $\CM_{fg}$ of finitely generated $\CM$-modules to create the \emph{Ringel quilt} of
this category. The term is borrowed from Ringel \cite{Rin11}, where it is called the \emph{Auslander--Reiten quilt}.

We start with the unbridged version of the $\AR$-quiver of the category $\CM_{fg}$, - see Figure 1. Now we
compactify (add the boundary to) this diagram taking into account the direct limits of rays, and inverse limits
of corays - see Figure \ref{fig4}.

\begin{figure}[t]
$$
\vcenter{%
\xymatrix@C=16pt@R=14pt{%
&&&&*+={\circ}\ar@{}+<-14pt,4pt>*{_{xy\in C}}\ar[dr]_{\subset}&&&&\\
&&&&&*+={\circ}\ar@{}+<20pt,2pt>*{_{xy\in D}}\ar[dr]_y&&&&\\
&&*+={\circ}\ar@{}+<-16pt,4pt>*{_{n\in M_1}}\ar[dr]\ar[dl]\ar@{}+<8pt,8pt>*{.}\ar@{}+<14pt,14pt>*{.}\ar@{}+<20pt,20pt>*{.}
&&&&*+={\circ}\ar@{}+<17pt,4pt>*{_{xy^2\in C}}\ar[dr]_{\subset}&\\
&*+={\circ}\ar@{}+<-16pt,4pt>*{_{n\in X_1}}\ar[dr]\ar[dl]&&
*+={\circ}\ar@{}+<-18pt,0pt>*{_{ny\in Y_2}}\ar[dl]\ar[dr]\ar@{}+<8pt,8pt>*{.}\ar@{}+<14pt,14pt>*{.}\ar@{}+<20pt,20pt>*{.}
&&&&*+={\circ}\ar@{}+<18pt,4pt>*{_{xy^2\in D}}\ar@{}+<8pt,-8pt>*{.}\ar@{}+<14pt,-14pt>*{.}\ar@{}+<20pt,-20pt>*{.}&\\
*+={\circ}\ar@{}+<-16pt,4pt>*{_{xy\in B}}\ar[dr]&&*+={\circ}\ar@{}+<-18pt,0pt>*{_{ny\in N_1}}\ar[dr]\ar[dl]&&
*+={\circ}\ar@{}+<-18pt,0pt>*{_{ny\in M_2}}\ar[dr]\ar[dl]\ar@{}+<8pt,8pt>*{.}\ar@{}+<14pt,14pt>*{.}\ar@{}+<20pt,20pt>*{.}
&&&&&&\\
&*+={\circ}\ar@{}+<-18pt,0pt>*{_{ny\in Y_1}}\ar[dl]\ar[d]\ar[dr]&&
*+={\circ}\ar@{}+<-18pt,0pt>*{_{ny\in X_2}}\ar[dr]\ar[dl]&
&*+={\circ}\ar@{}+<-18pt,0pt>*{_{ny^2\in Y_3}}\ar[dl]\ar@{}+<8pt,8pt>*{.}\ar@{}+<14pt,14pt>*{.}\ar@{}+<20pt,20pt>*{.}
\ar@{}+<8pt,-8pt>*{.}\ar@{}+<14pt,-14pt>*{.}\ar@{}+<20pt,-20pt>*{.}
&&&&*+={\circ}\ar@{}+<-16pt,0pt>*{_{1\in G_y}}\\
*+={\circ}\ar@{}+<-18pt,0pt>*{_{xy^2\in A}}\ar[dr]&
*+={\circ}\ar@{}+<-4pt,0pt>*{_{xy\in\ \ S}}\ar[d]&*+={\circ}\ar@{}+<18pt,0pt>*{_{ny\in M_1}}\ar[dr]\ar[dl]&&
*+={\circ}\ar@{}+<21pt,0pt>*{_{ny^2\in N_2}}\ar[dl]\ar@{}+<8pt,-8pt>*{.}\ar@{}+<14pt,-14pt>*{.}\ar@{}+<20pt,-20pt>*{.}
&&&&\\
&*+={\circ}\ar@{}+<-18pt,0pt>*{_{xy\in X_1}}\ar[dr]\ar[dl]&&
*+={\circ}\ar@{}+<20pt,0pt>*{_{ny^2\in Y_2}}\ar[dl]\ar@{}+<8pt,-8pt>*{.}\ar@{}+<14pt,-14pt>*{.}\ar@{}+<20pt,-20pt>*{.}
&&&&*+={\circ}\ar@{}+<-8pt,8pt>*{_{xy\in \wt R}}%
\ar[dl]_{\subset}\ar@{}+<8pt,8pt>*{.}\ar@{}+<14pt,14pt>*{.}\ar@{}+<20pt,20pt>*{.}
\ar@{-->} '[dr]+<-4pt,4pt> '[uurr]+<40pt,16pt> '[uuuuuuull]+<6pt,30pt> '[uuuuuuulll]+<2pt,2pt>
&&\\
*+={\circ}\ar@{}+<-18pt,0pt>*{_{xy^2\in B}}\ar[dr]&&
*+={\circ}\ar@{}+<21pt,0pt>*{_{ny^2\in N_1}}\ar[dl]\ar@{}+<8pt,-8pt>*{.}\ar@{}+<14pt,-14pt>*{.}\ar@{}+<20pt,-20pt>*{.}
&&&&*+={\circ}\ar@{}+<18pt,0pt>*{_{xy\in N}}\ar[dl]_{y}
\ar@{-->} '[ddrr]+<-12pt,14pt> '[uuurrr]+<60pt,16pt> '[uuuuuuu]+<16pt,40pt> '[uuuuuuul]+<2pt,2pt>
&&\\
&*+={\circ}\ar@{}+<18pt,0pt>*{_{ny\in Y_2}}\ar@{}+<8pt,-8pt>*{.}\ar@{}+<14pt,-14pt>*{.}\ar@{}+<20pt,-20pt>*{.}
\ar@{}+<0pt,-8pt>*{.}\ar@{}+<0pt,-14pt>*{.}\ar@{}+<0pt,-20pt>*{.}&&&&
*+={\circ}\ar@{}+<16pt,0pt>*{_{xy^2\in \wt R}}\ar[ld]_{\subset}&&&&\\
&&&&*+={\circ}\ar@{}+<20pt,0pt>*{_{xy^2\in N}}\ar@{}+<-4pt,-4pt>*{.}\ar@{}+<-8pt,-8pt>*{.}\ar@{}+<-12pt,-12pt>*{.}&&&&\\
&&&*+={\circ}\ar@{}+<20pt,-2pt>*{_{xy\in Q_R}}&&&
}}
$$
\caption{}\label{fig4}
\end{figure}

Namely if we identify $M_k$ with the submodule of $Q_S$ generated by $y$ and $n_k= xy(x+y)^{-k}$ that the composite
morphism $M_k\to M_{k+1}$ on the diagram  is the inclusion. Further the pointed module $xy\in \wt R$ is the
direct limit of this ray, - we add this pointed module to its end.

Going along the parallel ray $(X_1,n)\to (N_1,ny)\to \dots$ in the northeastern direction, we obtain pointed
module $(N,xy)$ as its direct limit. Add this module to the end of this this ray and attach the irreducible map
$\wt R\subset N$ of these pointed modules. Now repeat this procedure with all parallel northeastern rays, for
instance add irreducible maps $N\xr{y} \wt R$.

Further attach the point $1\in G_y$ to the end of the ray $(C,xy)\to (D,xy)\to (C,xy^2)\to \dots$, because this
pointed module is the corresponding direct limit. On the other hand we attach the same module to the the
beginning of the coray $\dots\to (\wt R,xy)\to (N,xy)$, because the pp-type of $1\in G_y$ is the  same as the
pp-type of the inverse limit of pointed modules in this coray.

Now add the pointed module $xy\in Q_R$ to the end of the ray $(\wt R,xy)\to (N, xy)\to (\wt R, xy^2)\to \dots$
heading in southwestern direction. We will also connect $\wt R$ with the module $C$ by the irreducible map
$\wt R\xr{x} C$; and connect the next module $N$ in the ray to the next module $D$ by the irreducible map
$N\xr{x} D$.

These four modules form the following commutative square whose edges are irreducible maps.

\vspace{2mm}

\begin{center}
$
\xymatrix@C=16pt@R=16pt{%
*+{\wt R}\ar[r]^{\subset}\ar[d]_x&*+{N}\ar[d]^x\\
*+{C}\ar[r]^{\subset}&*+{D}\\
}
$
\end{center}

\vspace{3mm}

Further the next morphism $N\xr{y}\wt R$ along the southeastern ray leads to a similar commutative square.

\vspace{2mm}

\begin{center}
$
\xymatrix@C=16pt@R=16pt{%
*+{N}\ar[r]^y\ar[d]_x&*+{\wt R}\ar[d]^x\\
*+{D}\ar[r]^y&*+{C}\\
}
$
\end{center}

\vspace{3mm}

Note that the sides of these squares change the direction on Figure \ref{fig4} --- from going southeastern for
the top maps to heading southwestern for bottom morphisms.

\begin{figure}[t]
$$
\hspace{2cm}
\vcenter{%
\xymatrix@C=14pt@R=12pt{%
&&&&&&&&&&&&&&&&\\
&&&&&&&&&&&&&&&&\\
&&&&&&&*+={\circ}\ar@{}+<0pt,10pt>*{_{G_y}}&&&&\\
&&&&&&&&&&&&\\
&&&&&*+={\circ}\ar@{}+<0pt,10pt>*{_D}\ar@{}+<6pt,6pt>*{.}\ar@{}+<10pt,10pt>*{.}\ar@{}+<14pt,14pt>*{.}&&&&
*+={\circ}\ar@{}+<0pt,10pt>*{_{\wt R}}\ar[dr]_{\subset}
\ar@{}+<-6pt,6pt>*{.}\ar@{}+<-10pt,10pt>*{.}\ar@{}+<-14pt,14pt>*{.}&&&&&\\
&&&&*+={\circ}\ar[ur]_{\subset}\ar@{}+<0pt,10pt>*{_C}&&&&&&*+={\bullet}\ar@{}+<0pt,10pt>*{_N}\ar[dr]_y
\ar@{=>} '[ur]+<-5pt,-3pt> '[uuuuulll]+<0pt,0pt> '[lllllllllldd]+<4pt,0pt> [lllllllllddd]&&&\\
&&&*+={\circ}\ar@{}+<2pt,12pt>*{_D}\ar[ur]_y&&&&
*+={\circ}\ar[dr]\ar@{}+<0pt,10pt>*{_{Y_3}}\ar@{}+<6pt,6pt>*{.}\ar@{}+<10pt,10pt>*{.}\ar@{}+<14pt,14pt>*{.}
\ar@{}+<-6pt,6pt>*{.}\ar@{}+<-10pt,10pt>*{.}\ar@{}+<-14pt,14pt>*{.}&&&&
*+={\circ}\ar@{}+<0pt,10pt>*{_{\wt R}}\ar[dr]_{\subset}
&&&\\
&&*+={\circ}\ar@{}+<0pt,10pt>*{_C}\ar[ur]_{\subset}&&&&
*+={\circ}\ar[dr]\ar[ur]\ar@{}+<0pt,10pt>*{_{M_2}}\ar@{}+<-6pt,6pt>*{.}\ar@{}+<-10pt,10pt>*{.}\ar@{}+<-14pt,14pt>*{.}&&
*+={\bullet}\ar[dr]\ar@{}+<0pt,10pt>*{_{N_2}}\ar@{}+<6pt,6pt>*{.}\ar@{}+<10pt,10pt>*{.}\ar@{}+<14pt,14pt>*{.}&&&&
*+={\circ}\ar@{}+<0pt,10pt>*{_N}\ar[dr]_y&\\
&*+={\bullet}\ar@{}+<2pt,12pt>*{_D}\ar[ur]_y\ar@{}+<-6pt,-6pt>*{.}\ar@{}+<-10pt,-10pt>*{.}\ar@{}+<-14pt,-14pt>*{.}&&
&&*+={\circ}\ar[dr]\ar[ur]\ar@{}+<0pt,10pt>*{_{Y_2}}\ar@{}+<-6pt,6pt>*{.}\ar@{}+<-10pt,10pt>*{.}\ar@{}+<-14pt,14pt>*{.}
&&
*+={\bullet}\ar[dr]\ar@2{->}[ur]\ar@{}+<0pt,10pt>*{_{X_2}}&&
*+={\circ}\ar[dr]\ar@{}+<0pt,10pt>*{_{Y_2}}\ar@{}+<6pt,6pt>*{.}\ar@{}+<10pt,10pt>*{.}\ar@{}+<14pt,14pt>*{.}&&&&
*+={\circ}\ar@{}+<0pt,10pt>*{_{\wt R}}\ar@{}+<6pt,-6pt>*{.}\ar@{}+<10pt,-10pt>*{.}\ar@{}+<14pt,-14pt>*{.}&&\\
&&&&*+={\circ}\ar[dr]\ar[ur]\ar@{}+<0pt,10pt>*{_{M_1}}
\ar@{}+<-6pt,6pt>*{.}\ar@{}+<-10pt,10pt>*{.}\ar@{}+<-14pt,14pt>*{.}&&
*+={\bullet}\ar[dr]\ar@{=>}[ur]\ar@{}+<0pt,10pt>*{_{N_1}}&&
*+={\circ}\ar[dr]\ar[ur]\ar@{}+<0pt,10pt>*{_{M_1}}&&
*+={\circ}\ar[dr]\ar@{}+<0pt,10pt>*{_{N_1}}\ar@{}+<6pt,6pt>*{.}\ar@{}+<10pt,10pt>*{.}\ar@{}+<14pt,14pt>*{.}&&&&&&&&&\\
&&&*+={\bullet}\ar[dr]\ar[ur]\ar@{=>}[r]\ar@{}+<0pt,10pt>*{_{Y_1}}
\ar@{}+<-6pt,-6pt>*{.}\ar@{}+<-10pt,-10pt>*{.}\ar@{}+<-14pt,-14pt>*{.}
\ar@{}+<-6pt,6pt>*{.}\ar@{}+<-10pt,10pt>*{.}\ar@{}+<-14pt,14pt>*{.}&
*+={\bullet}\ar@{=>}[r]\ar@{}+<0pt,6pt>*{_S}&
*+={\bullet}\ar[dr]\ar@{=>}[ur]\ar@{}+<0pt,10pt>*{_{X_1}}&&*+={\circ}\ar[dr]\ar[ur]\ar[r]\ar@{}+<0pt,10pt>*{_{Y_1}}&
*+={\circ}\ar[r]\ar@{}+<0pt,6pt>*{_S}&*+={\circ}\ar[dr]\ar[ur]\ar@{}+<0pt,10pt>*{_{X_1}}&&
*+={\circ}\ar@{}+<0pt,10pt>*{_{Y_1}}\ar@{}+<6pt,0pt>*{.}\ar@{}+<10pt,0pt>*{.}\ar@{}+<14pt,0pt>*{.}&&&&
*+={\circ}\ar@{}+<0pt,10pt>*{_{Q_R}}&&\\
&&&&*+={\circ}\ar[ur]\ar@{}+<0pt,-10pt>*{_A}\ar@{}+<-6pt,0pt>*{.}\ar@{}+<-10pt,0pt>*{.}\ar@{}+<-14pt,0pt>*{.}&&
*+={\circ}\ar[ur]\ar@{}+<0pt,-10pt>*{_B}&&
*+={\circ}\ar[ur]\ar@{}+<0pt,-10pt>*{_A}&&
*+={\circ}\ar[ur]\ar@{}+<0pt,-10pt>*{_B}\ar@{}+<6pt,0pt>*{.}\ar@{}+<10pt,0pt>*{.}\ar@{}+<14pt,0pt>*{.}&&
}}
$$
\caption{}\label{fig5}
\end{figure}

Gluing the boarder lines of this surface using just discussed irreducible morphisms we obtain the M\"obius stripe.
For instance (see Figure \ref{fig5}) starting from $S$ we may go through $X_1\to N_1\to\dots$ along the ray
heading northeastern, leave the first component of the $\AR$-quiver when reaching $N$, and then enter the
second component at point $D$. Now we could reenter the first component through a coray towards $Y_1$, hence
complete the whole revolution by entering $S$. It is easily checked that the resulting map amounts to
multiplication by $x$.

There is one misleading point in this description. The initial choice of the point $D$ to enter (after $N$) the
second component of the $\AR$-quiver was arbitrary, hence in general we will not get the closed path. Furthermore,
the module $G_x$ is missing from our description. It seems that a better way to give the global picture to the
category $\CM_{fg}$ is shown on Figure \ref{fig6}. Thus this category is the quotient of the vertical strip by the 
action given by the glide reflection shifting upwards.

\begin{figure}
$$
\vcenter{%
\xymatrix@C=6pt@R=6pt{%
&&&&*+{_{G_x}}\ar@{}+<0pt,-10pt>*{.}\ar@{}+<0pt,-16pt>*{.}\ar@{}+<0pt,-22pt>*{.}&&&\\
&&&&&&&\\
&&&&&&&\\
&&&&&&*+{_{G_y}}&\\
&&&&&*+{_D}\ar@{}+<8pt,6pt>*{.}\ar@{}+<12pt,9.5pt>*{.}\ar@{}+<16pt,13pt>*{.}&&*+{_{Q_R}}\ar[ul]\\
&&&&*+{_C}\ar[ur]&&*+{_N}\ar[ul]\ar@{}+<8pt,6pt>*{.}\ar@{}+<12pt,9.5pt>*{.}\ar@{}+<16pt,13pt>*{.}&\\
&&&*+{_D}\ar[ur]&&*+{_{\wt R}}\ar[ul]\ar[ur]&&*+{_A}\ar@{}+<-8pt,6pt>*{.}\ar@{}+<-12pt,9.5pt>*{.}\ar@{}+<-16pt,13pt>*{.}\\
&&*+{_C}\ar[ur]&&*+{_N}\ar@{=>}[ul]\ar[ur]&&*+{_{Y_1}}\ar[ur]
\ar@{}+<-8pt,6pt>*{.}\ar@{}+<-12pt,9.5pt>*{.}\ar@{}+<-16pt,13pt>*{.}&
\ar@{}+<30pt,0pt>*{}\\
&&&*+{_{\wt R}}\ar[ul]\ar[ur]&&*+{_{N_1}}\ar[ur]\ar@{}+<-8pt,6pt>*{.}\ar@{}+<-12pt,9.5pt>*{.}\ar@{}+<-16pt,13pt>*{.}&&
*+{_B}\ar[ul]\\
&&*+{_{G_y}}\ar@{}+<8pt,6pt>*{.}\ar@{}+<12pt,9.5pt>*{.}\ar@{}+<16pt,13pt>*{.}&&
*+{_{Y_2}}\ar[ur]\ar@{}+<-8pt,6pt>*{.}\ar@{}+<-12pt,9.5pt>*{.}\ar@{}+<-16pt,13pt>*{.}&&*+{_{X_1}}\ar[ur]\ar@{=>}[ul]&\\
&*+{_{Q_R}}\ar[ur]&&*+{_D}\ar@{}+<8pt,6pt>*{.}\ar@{}+<12pt,9.5pt>*{.}\ar@{}+<16pt,13pt>*{.}
\ar@{}+<-8pt,6pt>*{.}\ar@{}+<-12pt,9.5pt>*{.}\ar@{}+<-16pt,13pt>*{.}&&
*+{_{M_1}}\ar[ur]\ar[ul]&*+{_S}\ar@{=>}[u]&*+{_A}\ar[ul]\ar@{}+<0pt,-6pt>*{.}\ar@{}+<0pt,-10pt>*{.}\ar@{}+<0pt,-14pt>*{.}\\
&&*+{_N}\ar[ur]\ar@{}+<-8pt,6pt>*{.}\ar@{}+<-12pt,9.5pt>*{.}\ar@{}+<-16pt,13pt>*{.}&&
*+{_C}\ar[ul]\ar@{}+<8pt,6pt>*{.}\ar@{}+<12pt,9.5pt>*{.}\ar@{}+<16pt,13pt>*{.}&&
*+{_{Y_1}}\ar[ul]\ar@{=>}[u]\ar[ur]
\ar@{-->} '[d]+<-0pt,-0pt> '[r]+<10pt,10pt> '[uuuuuuur]+<10pt,4pt> '[uuuuuuuuu]+<0pt,-4pt> '[uuuulllll]+<14pt,0pt> [d]
&\\
&*+{_A}\ar@{}+<8pt,6pt>*{.}\ar@{}+<12pt,9.5pt>*{.}\ar@{}+<16pt,13pt>*{.}&&
*+{_{\wt R}}\ar[ul]\ar[ur]&&*+{_D}\ar[ul]\ar@{}+<8pt,6pt>*{.}\ar@{}+<12pt,9.5pt>*{.}\ar@{}+<16pt,13pt>*{.}&&\\
&&*+{_{Y_1}}\ar[ul]\ar@{}+<8pt,6pt>*{.}\ar@{}+<12pt,9.5pt>*{.}\ar@{}+<16pt,13pt>*{.}&&
*+{_N}\ar[ul]\ar@{=>}[ur]&&*+{_C}\ar[ul]&\\
&*+{_B}\ar[ur]\ar@{}+<-30pt,0pt>*{}&&*+{_{N_1}}\ar[ul]\ar@{}+<8pt,6pt>*{.}\ar@{}+<12pt,9.5pt>*{.}\ar@{}+<16pt,13pt>*{.}&&
*+{_{\wt R}}\ar[ul]\ar[ur]&&\\
&&*+{_{X_1}}\ar[ul]\ar@{=>}[ur]&&*+{_{Y_2}}\ar[ul]\ar@{}+<8pt,6pt>*{.}\ar@{}+<12pt,9.5pt>*{.}\ar@{}+<16pt,13pt>*{.}&&
*+{_{G_y}}\ar@{}+<-8pt,6pt>*{.}\ar@{}+<-12pt,9.5pt>*{.}\ar@{}+<-16pt,13pt>*{.}&\\
&*+{_A}\ar[ur]\ar@{}+<0pt,-6pt>*{.}\ar@{}+<0pt,-10pt>*{.}\ar@{}+<0pt,-14pt>*{.}&
*+{_S}\ar@{=>}[u]&*+{_{M_1}}\ar[ul]\ar[ur]&&
*+{_D}\ar@{}+<8pt,6pt>*{.}\ar@{}+<12pt,9.5pt>*{.}\ar@{}+<16pt,13pt>*{.}
\ar@{}+<-8pt,6pt>*{.}\ar@{}+<-12pt,9.5pt>*{.}\ar@{}+<-16pt,13pt>*{.}&&
*+{_{Q_R}}\ar[ul]\\
&&*+{_{Y_1}}\ar[ul]\ar@{=>}[u]\ar[ur]
\ar@{=} '[d]+<-0pt,-0pt> '[rrrruuu]+<16pt,20pt> '[uuuuuuuu]+<0pt,16pt> '[uuuuuuul]+<-14pt,4pt> '[ull]+<10pt,-4pt> [d]
&&*+{_C}\ar[ur]\ar@{}+<-8pt,6pt>*{.}\ar@{}+<-12pt,9.5pt>*{.}\ar@{}+<-16pt,13pt>*{.}&&
*+{_N}\ar[ul]\ar@{}+<8pt,6pt>*{.}\ar@{}+<12pt,9.5pt>*{.}\ar@{}+<16pt,13pt>*{.}&\\
&&&*+{_D}\ar[ur]\ar@{}+<-8pt,6pt>*{.}\ar@{}+<-12pt,9.5pt>*{.}\ar@{}+<-16pt,13pt>*{.}&&
*+{_{\wt R}}\ar[ul]\ar[ur]&&*+{_A}\\
&&*+{_C}\ar[ur]&&*+{_N}\ar@{=>}[ul]\ar[ur]&&*+{_{Y_1}}\ar[ur]
\ar@{}+<-8pt,6pt>*{.}\ar@{}+<-12pt,9.5pt>*{.}\ar@{}+<-16pt,13pt>*{.}&
\ar@{}+<30pt,0pt>*{}\\
&&&*+{_{\wt R}}\ar[ul]\ar[ur]&&*+{_{N_1}}\ar[ur]\ar@{}+<-8pt,6pt>*{.}\ar@{}+<-12pt,9.5pt>*{.}\ar@{}+<-16pt,13pt>*{.}&&
*+{_B}\ar[ul]\\
&&&&*+{_{Y_2}}\ar[ur]\ar@{}+<-8pt,6pt>*{.}\ar@{}+<-12pt,9.5pt>*{.}\ar@{}+<-16pt,13pt>*{.}&&*+{_{X_1}}\ar@{=>}[ul]\ar[ur]&\\
&&&&&*+{_{M_1}}\ar[ul]\ar[ur]&*+{_S}\ar@{=>}[u]&
*+{_A}\ar[ul]\ar@{}+<0pt,-6pt>*{.}\ar@{}+<0pt,-10pt>*{.}\ar@{}+<0pt,-14pt>*{.}\ar[ul]\\
&&&&&&*+{_{Y_1}}\ar[ul]\ar@{=>}[u]\ar[ur]
\ar@{--} '[d]+<-0pt,-0pt> '[r]+<10pt,10pt> '[uuuuuuur]+<10pt,4pt> '[uuuuuuuuu]+<0pt,-4pt> '[uuuulllll]+<14pt,0pt> [d]
&\\
&&&&&&&
}}
$$
\caption{}\label{fig6}
\end{figure}

This way we could add the point $G_x$ to the top of this diagram, because it is the direct limit of all
rays heading upwards. We will also mark on this figure the points of $\ZCM_S$ included into the basic open set
$(x^2\mid v/ v=0)$ corresponding to the first interval (see Section \ref{S-cut1}).

\section{The infinite radical}

In this section we will calculate the nilpotency index of the radical of the category $\CM_{fg}$ of finitely
generated $\CM$-modules over $S$. Recall that the \emph{radical}, $\rad$, of this category is defined to consist
of all morphism $f: M\to N$ between modules in this category such that for each indecomposable $K\in \CM_{fg}$,
any composite map $K\to M\xr{f} N\to K$ is not invertible. In follows easily that each irreducible morphism is
in the radical, in fact $\rad$ is generated by irreducible morphisms.

The \emph{transfinite powers} of the radical, $\rad^{\lam}$, are defined as follows. Set $\rad^1= \rad$. If $\lam$
is a limit ordinal, then define $\rad^{\lam}= \cap_{\mu< \lam} \rad^{\mu}$. Otherwise $\lam= \nu+ k$ for
a limit ordinal $\nu$, hence set $\rad^{\lam}= (\rad^{\nu})^{k+1}$. For instance $\rad^{\,\om}$ is known as
the \emph{infinite radical}.

The least ordinal $\lam$ such that $\rad^{\lam}=0$ (if exists) is called the \emph{nilpotency index of the radical},
$\nil(\rad)$. If it is defined, each morphism $f$ in $\CM_{fg}$ is assigned its \emph{nilpotency index}, $\nil(f)$,
as the least $\mu$ such that $f\in \rad^{\mu}\sm \rad^{\mu+1}$. Thus the nilpotency index of the radical is the
supremum of indices of morphisms between indecomposable modules.

Here is he main result of this section.

\begin{theorem}\label{nilp}
The nilpotency index of the radical of the category of finitely generated Cohen--Macaulay $S$-modules equals
$\om\cdot 2$.
\end{theorem}

Recall that each ordinal can be uniquely written in the form $\om^{\lam_1} n_1+ \dots+ \om^{\lam_k} n_k$ for some
ordinals $\lam_1>\lam_2> \dots> \lam_k$, and some natural numbers $n_i$. Thus $\om\cdot 2$ is the standard form
for $\om+ \om$.

One implication in this theorem is easy. Namely the 'one revolution' map $f: S\xr{x} S$ is in the infinite radical,
hence its power $f^{k+1}$ is a nonzero map in $\rad^{\om +k}$, as desired.

To prove the converse let us make the following remark. Suppose that $f: M\to N$ be a morphism between 
indecomposable modules in $\CM_{fg}$ which belongs to the infinite radical. If $M$ lies in the top component 
of the $\AR$-quiver on Figure \ref{fig4} (i.e. $M$ is $C$ or $D$) then, factoring through irreducible maps, we 
conclude that $f$ splits through a module in the second component. Similarly if $M$ is in the second component, then 
$f$ splits through the multiplication by $x$.

It follows that if $f\in \rad^{\,\om\cdot 2}$ then for each $m\in M$ its image $f(m)$ is divisible by any power 
of $x$. Looking at the Figure \ref{fig2} we conclude that $f(m)=0$, hence $f=0$. From this the result easily follows.


\end{document}